\documentclass[11pt,twoside, a4paper, english, reqno]{amsart}
\usepackage{amssymb,amsfonts,latexsym,color}
\usepackage{graphics,verbatim}
\usepackage{graphicx}

\newtheorem{theorem}{Theorem}[section]
\newtheorem{proposition}{Proposition}[section]
\newtheorem{lemma}{Lemma}[section]

\newtheorem{remark}{Remark}[section]
\newtheorem{definition}{Definition}[section]

 \newcommand{\<}{\left\langle}
\renewcommand{\>}{\right\rangle}

\newcommand{\eps}{\varepsilon}

\newcommand{\To}{\longrightarrow}

\newcommand{\Real}{\mathbb{R}}

\newcommand{\abs}[1]{\left\vert#1\right\vert}

\newcommand{\norm}[1]{\left\Vert#1\right\Vert}

\newcommand{\be} {\begin{equation}}
\newcommand{\ee} {\end{equation}}
\newcommand{\bea} {\begin{eqnarray}}
\newcommand{\eea} {\end{eqnarray}}
\newcommand{\Bea} {\begin{eqnarray*}}
\newcommand{\Eea} {\end{eqnarray*}}
\newcommand{\pa} {\partial}

\newcommand{\al} {\alpha}
\newcommand{\ba} {\beta}
\newcommand{\de} {\delta}

\newcommand{\ga} {\gamma}
\newcommand{\Ga} {\Gamma}
\newcommand{\Om} {\Omega}

\newcommand{\De} {\Delta}
\newcommand{\la} {\lambda}

\newcommand{\nequiv} {\not\equiv}

\newcommand{\no} {\nonumber}

\newcommand{\lab} {\label}
\newcommand{\va} {\varphi}
\newcommand{\var} {\varepsilon}
\newcommand{\f}{\frac}
\newcommand{\Z}{\mathbb Z}
\newcommand{\R}{\mathbb R}
\newcommand{\N}{\mathbb N}
\newcommand{\Rn}{\mathbb R^N}
\newcommand{\Iom}{\int_{\Omega}}
\newcommand{\deb}{\rightharpoonup}
\catcode`\@=11

\makeatletter \@addtoreset{equation}{section} \makeatother

\begin{document}

\title[Multiplicity and sign changing solutions]{Multiplicity results and sign changing solutions of non-local equations with concave-convex nonlinearities}
\author{Mousomi Bhakta, \ Debangana Mukherjee}
\address{Department of Mathematics, Indian Institute of Science Education and Research, Dr. Homi Bhaba Road, Pune-411008, India}
\email{mousomi@iiserpune.ac.in, \  debangana18@gmail.com}

\maketitle
\begin{abstract}
In this paper we prove the existence of infinitely many nontrivial solutions of the following equations driven by a nonlocal integro-differential operator $\mathcal{L}_K$ with concave-convex nonlinearities and homogeneous Dirichlet boundary conditions
\begin{eqnarray*}
 \mathcal{L}_{K} u + \mu |u|^{q-1}u + \la|u|^{p-1}u &=& 0 \quad\mbox{in}\quad \Omega,\no\\
  u&=&0 \quad\mbox{in}\quad\mathbb{R}^N\setminus\Omega,
\end{eqnarray*} 
where $\Om$ is a smooth bounded domain in $\Rn$, $N>2s$, $s\in(0, 1)$, $0<q<1<p\leq \f{N+2s}{N-2s}$. Moreover, when $\mathcal{L}_K$ reduces to the fractional laplacian operator $-(-\De)^s $, $p=\f{N+2s}{N-2s}$, $\f{1}{2}(\f{N+2s}{N-2s})<q<1$, $N>6s$, $\la=1$, we find $\mu^*>0$ such that for any $\mu\in(0,\mu^*)$, there exists at least one sign changing solution.
 \end{abstract}

\section{\bf Introduction}
In recent years, a great deal of attention has been devoted to fractional and non-local operators of elliptic type. One of the main reasons comes from the fact that
this operator naturally arises in several physical phenomenon like flames propagation and chemical reaction of liquids, population
dynamics, geophysical fluid dynamics, mathematical finance etc (see \cite{A, Be, CT, V, VIKH} and the references therein). In all these cases, the nonlocal effect was modelled by the singularity at infinity.

In this paper we mainly focus on the following problem with general integro-differential operator and concave-convex nonlinearities.
\begin{align*}
 \left(\mathcal{P}_{K}\right)
 \begin{cases}
  \mathcal{L}_{K} u + \mu |u|^{q-1}u + \la|u|^{p-1}u=0 &\quad\mbox{in}\quad \Omega,\\
  u=0&\quad\mbox{in}\quad\mathbb{R}^N\setminus\Omega,
 \end{cases}
\end{align*}
where $s\in(0,1)$ is fixed, $N>2s$,  $\Omega$ is an open, bounded domain with smooth boundary, $0<q<1, \ 1 <p \leq 2^*-1$ with $2^*=\frac{2N}{N-2s}$.  Here $\mathcal{L}_{K}$ is the non-local operator defined as follows:
\begin{align} \label{L_K}
 \mathcal{L}_{K}u(x)=\frac{1}{2}\int_{\mathbb{R}^N}\left(u(x+y)+u(x-y)-2u(x)\right)K(y)dy,\,\,\,x\in\mathbb{R}^N. 
\end{align}
Here $K:\mathbb{R}^N\setminus\{0\}\to (0,+\infty)$ is a function defined such that
\begin{align}\label{assum-1}
 mK(x) \in L^1(\Rn)\quad\mbox{with}\quad m(x)=\min\{|x|^2,1\};
\end{align}
\begin{align}\label{assum-2}
 \mbox{there exists}\quad\theta >0\quad\mbox{such that}\quad K(x) \geq \theta |x|^{-(N+2s)}\ \mbox{for any}\ 
 x \in \Rn \setminus \{0\};
\end{align}
\begin{align}\label{assum-3}
\mbox{and}\quad K(x)=K(-x)\quad\mbox{for any}\quad x \in \Rn \setminus \{0\}. 
\end{align}
A model for $K$ is given by $K(x)=|x|^{-(N+2s)}. $ In this case $L_K$ reduces to the fractional Laplace operator $-\left(-\Delta\right)^s$,
defined below up to a normalization constant 
\begin{align} \label{De-u}
  -\left(-\Delta\right)^s u(x)=\frac{1}{2}\int_{\mathbb{R}^N}\frac{u(x+y)+u(x-y)-2u(x)}{|y|^{N+2s}}dy,\,\,\,x\in\mathbb{R}^N.
\end{align}
By $X$ we denote the linear space of Lebesgue measurable functions from $\Rn$
to $\R$ such that if $g\in X$ then $g|_{\Omega}\in L^2(\Omega)$ and
   $$(g(x)-g(y))\sqrt{K(x-y)} \in L^2(Q,dxdy),$$
  where $Q=\R^{2N}\setminus (C\Om \times C\Om)$ with $C\Om=\Rn \setminus \Om. $
   The space $X$ is endowed with the norm defined:
   $$\norm{u}_X=\norm{u}_{L^2(\Om)}+\left(\int_Q |u(x)-u(y)|^2K(x-y)dxdy\right)^{1/2}.$$
Then we define $X_0 :=\Big\{u \in X:u=0 \quad\text{a.e. in}\quad \Rn \setminus \Om\Big\} $ with the norm   
   $$\norm{u}_{X_0}=\left(\int_Q|u(x)-u(y)|^2K(x-y)dxdy\right)^{1/2}.$$
   With this norm, $X_0$ is a Hilbert space with the scalar product
   $$\<u,v\>_{X_0}=\int_Q (u(x)-u(y))(v(x)-v(y))K(x-y)dxdy, $$
(see \cite[lemma 7]{SerVal3}). For further details on $X$ and $X_0$ and for their properties, we refer to \cite{NePaVal} and the references therein. 
      Thanks to \eqref{assum-1},  it can be shown that $C_0^2(\Om) \subseteq X_0$, see \cite[Lemma 11]{SerVal4} and so $X$ and 
   $X_0$ are non-empty .
   \begin{definition}\label{def-sol} We say that $u\in X_0$ is a weak solution of 
   $(\mathcal{P}_K)$ if 
   \begin{eqnarray*} 
 \int_{\R^{2N}}(u(x)-u(y))(\phi(x)-\phi(y))K(x-y)dxdy &=& \mu\Iom |u(x)|^{q-1}u(x)\phi(x)dx \\
 &+& \la\Iom |u(x)|^{p-1}u(x)\phi(x)dx
 \end{eqnarray*}
 for all $\phi \in X_0. $
 \end{definition}
   
   We denote by $H^s(\Rn)$ the usual fractional Sobolev space endowed with the so-called Gagliardo norm
      $$\norm{g}_{H^s(\Rn)}=\norm{g}_{L^2(\Rn)}+\bigg(\int_{\Rn \times \Rn} \frac{|g(x)-g(y)|^2}{|x-y|^{N+2s}}dxdy\bigg)^{1/2}.$$ 
   When $K(x)=|x|^{-(N+2s)}$, $X_0$ reduces to $\Big\{u \in H^s(\Rn):u=0 \quad\text{a.e. in}\quad \Rn \setminus \Om\Big\} $.
  When $\Om$ is bounded, the norm in $X_0(\Om)$ is equivalent to 
  $$\norm{u}_{X_0(\Om)}=\bigg(\int_{\Rn \times \Rn} \frac{|u(x)-u(y)|^2}{|x-y|^{N+2s}}dxdy\bigg)^{1/2}, $$
  (see \cite{BCSS}).  By \cite[Proposition 3.6]{NePaVal} we have,
 \be\lab{norm-2} 
 \norm{u}_{X_0(\Om)}=\norm{\left(-\Delta\right)^{s/2} u}_{L^2(\Rn)}.
 \ee
 The Euler-Lagrange energy functional associated to  $(\mathcal{P}_K$) is
 \be\lab{I-mu-la}
 I_\mu^\la(u)=\frac{1}{2}\int_{\R^{2N}}|u(x)-u(y)|^2K(x-y)dxdy -\frac{\mu}{q+1}\Iom|u|^{q+1}dx-\frac{\la}{p+1}\Iom|u|^{p+1}dx. 
 \ee
Thanks to the Sobolev embedding $X_0\hookrightarrow L^{2^*}(\Rn)$ (see \cite[Lemma 9]{SerVal2}), $I_\mu^\la$ is well defined $C^1$ functional on $X_0$. It is well known that there exists a one-to-one correspondence between the weak
solutions of $(\mathcal{P}_K)$ and the critical point of $I_{\mu}^{\la}$ on $X_0$. We define the best fractional critical Sobolev constant $S_k$ as
\be\lab{S-K}
S_K := \inf_{v \in X_0 \setminus \{0\}}\frac{\displaystyle\int_{\R^{2N}}|v(x)-v(y)|^2K(x-y)dxdy}{\left(\displaystyle\Iom |v(x)|^{2^*}\right)^{2/2^*}}.
\ee

 A classical topic in nonlinear analysis is the study of existence and multiplicity of solutions for nonlinear equations. There are many results on the subject of concave-convex nonlinearity involving different local and nonlocal operators. Elliptic problems in bounded domains involving concave and convex terms have been studied
extensively since Ambrosetti, Brezis and Cerami \cite{ABC} considered the following equation:
\begin{align*}
 \left(E_{\mu}\right)
\begin{cases}
-\De u =&\mu u^{q-1}+u^{p-1} \quad\text{in}\quad\Om,\\
\,\, \,\,\,\,\,\,\, u >&0 \quad\text{in}\quad\Om,\\
\,\, \,\,\,\,\,\,\, u =&0 \quad\text{on}\quad\partial\Om,
\end{cases}
\end{align*}
where $1<q<2<p\leq \f{2N}{N-2}$, $\mu>0$ and $\Om$ is a bounded domain in $\Rn$. They found that there exists $\mu_0>0$ such that ($E_{\mu}$) admits at least two positive solutions for $\mu\in(0,\mu_0)$, one positive solution for  $\mu=\mu_0$ and no positive solution exists for $\mu>\mu_0$ (see also Ambrosetti, Azorero and Peral \cite{AAP} for more references therein). Later on Adimurthi-Pacella-Yadava \cite{APY},  Damascelli,
Grossi and Pacella \cite{DGP}, Ouyang and Shi \cite{OS} and Tang \cite{Tang} proved there exists $\mu_0>0$ such that for $\mu\in(0,\mu_0)$, there are exactly two positive solutions of $(E_{\mu})$ when $\Om$ is the unit ball in $\Rn$ and  exactly one positive solution for $\mu=\mu_0$ and no positive solution exists for $\mu>\mu_0$.
For  the local operator we also quote \cite{ BM, BW, CCP, Chen, GP,  Wu} and the references therein.  In past couple of years many of these results have been generalised to the case of  nonlocal operators,  we refer a few among them \cite{BCSS, BrCPS, DMV, MRS, PP} and the references therein. We also quote here a very important paper by Chen, Li and Ou \cite{CLO}, where the authors have classified all the positive solutions of the fractional Yamabe equation. As per our knowledge no result for sign changing solution involving non-local operator and concave-convex nonlinearity has been studied so far. 
 
 \vspace{2mm}
 
 The main results of our paper are stated below. First we study the critical case $p=2^*-1, \la=1, $ that is, 
 \begin{align*}
 \left(\mathcal{P'}_{K}\right)
 \begin{cases}
  \mathcal{L}_{K} u + \mu |u|^{q-1}u + |u|^{2^*-2}u=0 &\quad\mbox{in}\quad \Omega,\\
  u=0&\quad\mbox{in}\quad\mathbb{R}^n\setminus\Omega.
 \end{cases}
\end{align*}
\begin{theorem} \label{thm.1}
 Let $\Om$ be a bounded domain in $\Rn$ with smooth boundary, $N>2s$. Then there exists $\mu^*>0$ such that for all $\mu \in (0,\mu^*)$,
 problem $(\mathcal{P'}_K)$ has a sequence of non-trivial solutions $\{u_n\}_{n \geq 1}$ such that $I(u_n) <0$ and $I(u_n) \to 0$ as $n \to \infty$
 where $I(.)$ is the corresponding energy functional associated with $(\mathcal{P'}_K). $
\end{theorem}

\begin{remark}
Here we would like to mention that when $K(x)=|x|^{-(N+2s)}$,  it has been proved in  \cite{BCSS} that there exists $\Lambda>0$ such that, $(\mathcal{P'}_K)$ has at least  two positive solutions when $\mu\in(0,\Lambda)$, no positive solution when $\mu>\Lambda$ and at least one positive solution when $\mu=\Lambda$.  Chen-Deng  \cite{CD}    have proved that $(\mathcal{P'}_K)$ has at least two positive solutions when $\mu\in(0,\mu_0)$ for some $\mu_0>0$ under the assumption that 
\be\label{con:ex}
\text{There exists}\ u_0\in X_0 \ \text{with} \ u_0\geq 0 \ \text{a.e. in} \  \Om, \ \text{such that} \     \sup_{t\geq 0} I(tu_0)<\f{s}{N}S_K^\f{N}{2s}.
\ee
When $K(x)=|x|^{-(N+2s)}$, condition \eqref{con:ex} can be guaranteed by results of \cite{SerVal2}.
\end{remark}

The most important theorem in this paper is the following one, where we establish 
 existence of at least one sign changing solution of $(\mathcal{P'}_K)$ when $K(x)=|x|^{-(N+2s)}$, i.e., $\mathcal{L}_{K}=-(-\De)^s$,  under suitable assumption on $N$ and $q$.
\begin{theorem} \label{thm.2}
Let $\Om$ be a bounded domain with smooth boundary in $\Rn. $ Assume  $K(x)=|x|^{-(N+2s)}, s \in (0,1), N>6s, \   \f{1}{2}\big(\f{N+2s}{N-2s}\big) <q<1. $ Then
there exists $\mu^* >0$ such that for all $\mu \in (0,\mu^*)$ problem $(\mathcal{P'}_K)$ has at least one sign changing solution.
\end{theorem}

In the succeeding theorem, we prove the existence of infinitely many nontrivial solutions in the subcritical case. 
\begin{theorem} \label{thm.3}
Let $\Om$ be a bounded domain in $\Rn$ with smooth boundary, $N>2s, s \in (0,1). $ Assume $1<p<2^*-1. $ Then
\begin{enumerate}
 \item [(a)]
 For all $\la>0, \ \mu \in \R, \  (\mathcal{P}_K)$ has a sequence of nontrivial solutions $\{u_k\}_{k \geq 1}$ such that $I_\mu^\la(u_k) \to \infty$ as 
 $k \To \infty. $ Furthermore, if $\la >0,\ \mu \geq 0, $ then $\norm{u_k}_{X_0} \to \infty$ as $k \to \infty. $
 \item [(b)]
 For all $\mu>0,\  \la \in \R, \ (\mathcal{P}_K)$ has a sequence of nontrivial solutions $\{v_k\}_{k \geq 1}$ such that $I_\mu^\la(v_k) \to 0$ as 
 $k \to \infty. $ Furthermore, if $\mu >0,\ \la \leq 0, $ then $\norm{v_k}_{X_0} \to 0$ as $k \to \infty. $
\end{enumerate}
\end{theorem}

\begin{remark}
When $K(x)=|x|^{-(N+2s)}$,  Brandle, et. al \cite{BrCPS} have proved that there exists $\Lambda>0$ such that, $(\mathcal{P}_K)$ has at least  two positive solutions when $\mu\in(0,\Lambda)$, one positive solution when $\mu=\Lambda$ and  no positive solution when $\mu>\Lambda$. For general $K$ satisfying assumptions \eqref{assum-1}-\eqref{assum-3}, Chen-Deng \cite{CD} have proved that there exists at least two positive solutions of $(\mathcal{P}_K)$ when $\la=1$ and $\mu\in(0,\mu_0)$ for some $\mu_0>0$.
\end{remark}

To prove 
infinitely many nontrivial solutions of the above stated problems, we apply the Fountain Theorem and the Dual Fountain theorem which were  proved by Bartsch  \cite{B}  and Bartsch-Willem  \cite{BW} respectively (also see \cite{W}).  As usual for critical point theorems, we need to study the compactness properties of the functional together with its geometric features. With respect to the compactness, we need to prove that the functional satisfies
the classical Palais-Smale (PS)$_c$ assumption. But observe that  $X_0\hookrightarrow L^{2^*}(\Om)$ is not compact (see \cite[Lemma 9-b]{SerVal2}). Hence the  (PS)$_c$ condition does not hold globally for all $c$ and we have to prove that the energy level of the corresponding energy functional lies below the threshold of application of the (PS)$_c$ condition. The proof of Theorem \ref{thm.2} is much delicate. There by using decomposition 
of  Nehari manifold, we have estimated the energy and proved the existence of at least one sign changing solution by extending the result dealt within \cite{Chen}  for the classical Laplacian case. The proof is divided into several steps.\\

The paper is organised as follows. Section 2 is devoted to some notations and preliminaries about Fountain Theorem and Dual Fountain Theorem. In Section 3 we prove Theorem \ref{thm.1}. In Section 4 we deal with sign changing solution, namely we study Theorem \ref{thm.2}.  In section 5 we study the subcritical problem, Theorem \ref{thm.3}, Finally, in Section 6  we state results for a related problem that
can be solved using our methods as in Theorem \ref{thm.1} and we also point out some related open questions in Section 8. Section 7 is appendix. \\

{\bf Notations}: Throughout the paper $C$  denotes a general positive constant which may vary from line to line.

\section{\bf Preliminaries}

We start this section by recalling two abstract theorems namely the Fountain theorem and the Dual Fountain Theorem. For this, we need some definitions from \cite{W}.
\begin{definition}
The action of a topological group $G$ on a Banach space $X$ is a continuous map
$$G\times X\To X : [g, u]\To gu, $$ such that 
$$1.u=u, \quad (gh)u=g(hu), \quad u\mapsto gu \quad\text{is linear}.$$
The action is isometric if $||gu||=||u||$. The space of invariant points is defined by 
$$\text{Fix}\ (G):= \{u\in X: gu=u \quad\forall g\in G\}.$$
A set $A\subset X$ is called invariant if $gA=A$ for every $g\in G$. A functional $\va: X\To\R$ is called invariant if $\va\circ g=\va$ for every $g\in G$.  A map $f: X\To X$ is called equivariant if  $g\circ f=f\circ g$ for every $g\in G$.
\end{definition}

\begin{definition}
Let $G$ be a compact group on Banach space $X_0$. Assume that $G$ acts diagonally on $V^k$
$$g(v_1, \cdots, v_k):= (gv_1, \cdots, gv_k) ,$$ where $V$ is a finite dimensional space. The action of $G$ is admissible if every continuous equivariant map $\pa U\To V^{k-1}$, where $U$ is an open bounded invariant neighborhood of $0$ in $V^k$, $k\geq 2$, has a zero. 
\end{definition}
By \it{Borsuk-Ulam} Theorem, the antipodal action of $G:= \Z/2$ on $V:=\R$ is admissible (see \cite[Theorem D.17]{W}).

\vspace{2mm}

We consider the following situation:

\vspace{2mm}

\bf{(A1)} \it{The compact group $G$ acts isometrically on the Banach space $X= \overline{\bigoplus_{j\in \N} X_j}$, the spaces $X_j$ are invariant and there exists a finite dimensional space $V$ such that, for every $j\in \N$, $X_j\simeq V$ and the action of $G$ on $V$ is admissible.}

\begin{definition}
Let $\va\in C^1(X, \R)$. We say that $\{u_n\}$ is a Palais-Smale sequence (in short, PS sequence)  of $\va$ at level $c$ if $\va(u_n)\to c$ and $\va'(u_n)\to 0$ in $(X)'$, the dual space of $X$. Moreover,  we say that 
$\va$ satisfies (PS)$_c$ condition if $\{u_n\}$ is any (PS) sequence in $X$ at level $c$ implies $\{u_n\}$ has a convergent subsequence in $X$.
\end{definition}

\begin{theorem}\lab{t:F}[Fountain Theorem, Bartsch, 1993]
Under the assumption (A1), let $\va\in C^1(X, \R)$ be an invariant functional. If, for every $k\in \N$, there exists $0<r_k<\rho_k$ such that
\begin{itemize}
\item[(A2)] $a_k:= \max_{u\in Y_k, ||u||=\rho_k} \va(u)\leq 0$,
\item[(A3)] $b_k:= \inf_{u\in Z_k, ||u||=r_k} \va(u)\to\infty \quad{as}\quad k\to\infty$.
\item[(A4)] $\va$ satisfies (PS)$_c$ condition for every $c>0$, 
\end{itemize}
then $\va$ has an unbounded sequence of critical values.
\end{theorem}
\begin{theorem}\lab{t:DF}[Dual Fountain Theorem, Bartsch-Willem, 1995]
Under the assumption (A1), let $\va\in C^1(X, \R)$ be an invariant functional. If, for every $k\geq k_0$, there exists $0<r_k<\rho_k$ such that
\begin{itemize}
\item[(D1)] $a_k:= \inf_{u\in Z_k, ||u||=\rho_k} \va(u)\geq 0$,
\item[(D2)] $b_k:= \max_{u\in Y_k, ||u||=r_k} \va(u)< 0$,
\item[(D3)] $d_k:= \inf_{u\in Z_k, ||u||\leq\rho_k} \va(u)\to 0 \quad{as}\quad k\to\infty$.
\item[(D4)]  For every sequence $u_{r_{j}}\in X$ and $c\in [d_k, 0)$  such that 
$$u_{r_{j}}\in Y_{r_{j}}, \quad \va(u_{r_{j}})\to c \quad\text{and}\quad \va|_{Y_{r_{j}}}'(u_{r_{j}})\to 0 \quad\text{as}\quad r_j\to\infty,$$ contains a subsequence converging to a critical point of $\va$,
\end{itemize}
then $\va$ has a sequence of negative critical values converging to $0$.
\end{theorem}

We choose an orthonormal basis $\{e_j\}_{j=1}^{\infty}$ of $X_0$ (see \cite{SerVal1}). Next, we consider the antipodal action of $G:= \Z/2$. Define
\begin{equation}\label{YZ_k}
X_j=\Real e_j, \quad Y_k:=\bigoplus_{j=1}^k X_j, \quad Z_k:=\overline{\bigoplus_{j=k}^{\infty} X_j}.
\end{equation}

\begin{lemma}\label{beta_k}
If $1\leq p+1<2^*$, then we have that  $$\beta_k:=\sup_{{u \in Z_k},{\norm{u}_{X_0}=1}} |u|_{L^{p+1}(\Om)}\to 0 \quad\text{as}\quad k \to \infty.$$ 
\end{lemma}
\begin{proof}
Clearly, $0<\beta_{k+1} \leq \beta_k$. Thus there exists $\ba\geq 0$, such that $\lim_{k\to \infty}\beta_k=\ba$. By the definition of $\beta_k$, for every $k \geq 1$, there exists $u_k \in Z_k$ such that $\norm{u_k}=1$ and
$|u_k|_{L^{p+1}(\Om)} > \frac{\beta_k}{2}$. Using the definition of $Z_k$, it follows $ u_k \deb 0$ in $X_0$.
Therefore Sobolev embedding implies $u_k \to 0$ in $L^{p+1}(\Om)$ and this completes the proof.
\end{proof}

\section{\bf Proof of Theorem \ref{thm.1}}

\proof The energy functional associated to ($\mathcal{P'}_K)$ is the following
 \bea\label{eq:I}
 I(u) &=&\frac{1}{2}\int_{\R^{2N}}|u(x)-u(y)|^2K(x-y)dxdy -\frac{\mu}{q+1}\Iom|u|^{q+1}dx-\frac{1}{2^*}\Iom|u|^{2^*}dx\no\\
 &=&\f{1}{2}||u||^2_{X_0}-\frac{\mu}{q+1}\Iom|u|^{q+1}dx-\frac{1}{2^*}\Iom|u|^{2^*}dx,
\eea
where $\mu>0$.
We will show that $I$ satisfies all the assumptions of Theorem \ref{t:DF}. $X_j, Y_j, Z_j$ are chosen as in \eqref{YZ_k} and $G:=\Z/2$. Therefore (A1) is satisfied.

\vspace{4mm}

Next to check (D1) holds,  we define 
$$\beta_k:=\sup_{{u \in Z_k},{\norm{u}_{X_0}=1}} |u|_{L^{q+1}(\Om)}, \quad c:=\sup_{u \in X_0, \norm{u}_{X_0}=1}|u|^{2^*}_{L^{2^*}(\Om)} \ \text{and}\ R:=\bigg(\f{2^*}{4c} \bigg)^\f{1}{2^*-2}.$$ It is easy to see, $||u||_{X_0}\leq R$ implies $\f{c}{2^*}||u||_{X_0}^{2^*}\leq \f{1}{4}||u||^2_{X_0}$. 
Therefore  for $u \in Z_k,\  \norm{u}_{X_0} \leq R$, we have
\bea\label{3-1}
I(u)&\geq& \frac{\norm{u}_{X_0}^2}{2}-\frac{\mu}{q+1}\beta_k^{q+1}\norm{u}_{X_0}^{q+1}-\frac{c}{2^*}\norm{u}_{X_0}^{2^*} \no\\
&\geq& \frac{\norm{u}_{X_0}^2}{4}-\frac{\mu}{q+1}\beta_k^{q+1}\norm{u}_{X_0}^{q+1}
\eea 
Choose $\rho_k:=(\frac{4\mu \beta_k^{q+1}}{q+1})^{\frac{1}{1-q}}. $
Using Lemma \ref{beta_k}, we see that $\ba_k \to 0$ as $k \to \infty$. As a consequence $\rho_k\to 0$.
Thus for $k$ large, $u \in Z_k$ and $\norm{u}_{X_0}=\rho_k$ we have  $I(u) \geq 0$ and (D1) holds true.

\vspace{4mm}

To see (D2) holds we note that $Y_k$ is finite dimensional and in finite dimensional space all the norms are equivalent. Therefore (D2) is satisfied if we choose $r_k>0$ small enough (since $\mu>0$) and therefore we can choose $r_k=\f{\rho_k}{2}$.

\vspace{4mm}

For $k$ large, $u \in Z_k, \norm{u}_{X_0} \leq \rho_k, $ we have from \eqref{3-1} that $d_k \geq \frac{-\mu}{q+1}\beta_k^{q+1} \rho_k^{q+1}$. On the other hand as $\mu>0$ from the definition of $I(u)$ it follows $I(u)\leq \frac{\rho_k^2}{2}$. Thu $d_k \leq \frac{1}{2}\rho_k^2$
Using both upper and lower bounds of $d_k$ and Lemma \ref{beta_k}, we see that (D3) is also satisfied. 

\vspace{4mm}

To check the assertion (D4), we consider a sequence $\{u_{r_j}\} \subset X_0$ such that as 
\be\lab{3:urj}
 \{u_{r_j}\} \in Y_{r_j},  \quad
I(u_{r_j}) \to c, \quad I'|_{Y_{r_j}}(u_{r_j}) \to 0 \quad\text{as}\quad r_j \to \infty.
\ee

\vspace{2mm}

{\bf Claim:} There exists $k>0$ such that  if  $\mu>0$ is arbitrarily chosen and 
\be\lab{3:c}
c<\f{s}{N}S_K^\f{N}{2s}-k\mu^\f{2^*}{2*-q-1}, 
\ee
then $\{u_{r_j}\}$ contains a subsequence converging to a critical point of $I$, where $\{u_{r_j}\}$ is as in \eqref{3:urj}.\\

Assuming the claim, first let us complete the proof. Towards this, we choose $\mu^*=\bigg(\f{sS_{K}^\f{N}{2s}}{Nk}\bigg)^\f{2^*-q-1}{2^*}$. Then $\mu\in(0, \mu^*)$ implies $\f{s}{N}S_K^\f{N}{2s}>k\mu^\f{2^*}{2*-q-1}$. Thus, if $c\in[d_k, 0)$ then we have
$$c<0<\f{s}{N}S_K^\f{N}{2s}-k\mu^\f{2^*}{2*-q-1}.$$ Hence applying the above claim, we see that (D4) holds true. Therefore the result follows by Theorem \ref{t:DF}. 

\vspace{2mm}

Here we prove the claim dividing into four steps. 

\vspace{2mm}

{\it Step 1}: $\{u_{r_j}\}$  is bounded in $X_0$.\\
This follows by standard arguments.  More precisely, since $I(u_{r_j})=c+o(1)$ and $\<I'(u_{r_j}),u_{r_j}\>=o(1)||u_{r_j}||_{X_0}$, computing  $ I(u_{r_j})-\frac{1}{2}\<I'(u_{r_j}),u_{r_j}\>$ , we get $|u_{r_j}|_{L^{2^*}(\Om)}^{2^*}\leq C_1+ \norm{u_{r_j}}_{X_0}o(1)+ C_2 |u_{r_j}|^{q+1}_{L^{q+1}(\Om)}$. Therefore using the definition of $I$ along with Sobolev inequality yields
$$\norm{u_{r_j}}^2_{X_0}\leq C\displaystyle\left[1+||u_{r_j}||_{X_0}o(1)+||u_{r_j}||^{q+1}_{X_0}\right]$$ and hence the boundedness follows. Therefore passing to a subsequence
if necessary we may assume $u_{r_j}\deb u$ in $X_0$, $u_{r_j}\to u$ in $L^{\ga}(\Rn)$ for $1\leq\ga<2^*$  and point-wise.

\vspace{2mm}

{\it Step 2}:  $\{u_{r_j}\}$ is a PS sequence in $X_0$ at level $c$, where $c$ is as in \eqref{3:c}.

%$\{I'(u_{r_j})\}_{j \geq 1}$ is bounded in $(X_0)'$, which is  the dual space of $X_0$. \\
To see this, let $v \in X_0$ be arbitrarily chosen. Then
\be\lab{new1}
\<I'(u_{r_j}),v\>=\<u_{r_j},v\>-\Iom {|u_{r_j}|^{2^*-2}u_{r_j}v}dx-\mu\Iom |u_{r_j}|^{q-1}u_{r_j}vdx.
\ee
Therefore, using Sobolev inequality and Step 1 we have,
\Bea
|\<I'(u_{r_j}),v\>| &\leq& \norm{u_{r_j}}_{X_0}\norm{v}_{X_0} +\Iom {|u_{r_j}|^{2^*-1}|v|} dx + \mu\Iom |u_{r_j}|^q|v| dx \\
                    &\leq& \norm{u_{r_j}}_{X_0}\norm{v}_{X_0}+c_1\norm{u_{r_j}}_{X_0}^{2^*-1}\norm{v}_{X_0}+c_2\mu\norm{u_{r_j}}_{X_0}^q\norm{v}_{X_0} \\
                    &\leq& (\norm{u_{r_j}}_{X_0}+c_1\norm{u_{r_j}}_{X_0}^{2^*-1}+c_2\mu\norm{u_{r_j}}^q_{X_0})\norm{v}_{X_0} \\
                    &\leq& C\norm{v}_{X_0}, 
\Eea
which in turn implies $\norm{I'(u_{r_j})}_{(X_0)'} \leq M$ for all $j \geq1$. 

By the definition of $Y_{r_j}$,  there exists a sequence $(v_{r_j}) \in Y_{r_j}$ such that
$v_{r_j} \to v$ in $X_0$ as $r_j \to \infty$. Thus 
\Bea
|\<I'(u_{r_j}), v\>| &\leq& |\<I'(u_{r_j}),v_{r_j}\>|+|\<I'(u_{r_j}),v-v_{r_j}\>| \\
   &\leq& \norm{I'|_{Y_{r_j}}(u_{r_j})}_{( X_0)'}\norm{v_{r_j}}_{X_0}+\norm{I'(u_{r_j})}_{( X_0)'}\norm{v-v_{r_j}} _{X_0}.
\Eea
Combining the hypothesis $I'|_{Y_{r_j}}(u_{r_j}) \to 0$ as $r_j \to \infty$ (see \eqref{3:urj}), Step 1 and the fact that $\{I'(u_{r_j})\}$ is uniformly bounded, we have
$|\<I'(u_{r_j}),v\>| \to 0$ as $r_j \to \infty$. This in turn implies that $\{u_{r_j}\}$ is a PS sequence in $X_0$ at level $c$, where $c$ is as in \eqref{3:c}.

\vspace{2mm}

{\it Step 3}:  $u$ satisfies $(\mathcal{P'}_K)$.\\
 Using Vitali's convergence theorem via H\"older inequality and Sobolev inequality, it is not difficult to check that  we can pass the limit $r_j\to\infty$  in \eqref{new1},  Thus we  obtain $\<I'(u), v\>=0$ for every $v$ in $X_0$. Hence, $\mathcal{L}_K u+ \mu |u|^{q-1}u + |u|^{2^*-2}u =0 \quad\mbox{in}\quad \Om. $ 

\vspace{2mm}

{\it Step 4}: Define $v_{r_j} := u_{r_j}-u$. Then it is not difficult to see that, 
\be\lab{3:3}
\norm{v_{r_j}}^2_{X_0}=\norm{u_{r_j}}^2_{X_0} - \norm{u}^2_{X_0}+o(1).
\ee
 On the other hand, by Brezis-Lieb lemma, we have 
 \be\lab{3:4}
 |u_{r_j}|^{2^*}_{L^{2*}(\Om)}=|v_{r_j}|^{2^*}_{L^{2*}(\Om)}+ |u|^{2^*}_{L^{2*}(\Om)}+o(1).
 \ee
 Therefore by doing a straight forward computation and using  $I(u_{r_j})\to c$, we get 
\be\lab{3:5}
I(u) + \frac{1}{2}\norm{v_{r_j}}^2_{X_0}-\frac{1}{2^*}|v_{r_j}|^{2^*}_{L^{2^*}(\Om)}\to c.
\ee
 Since  $\<I'(u_{r_j}),u_{r_j}\>\to 0$ and $\<I'(u),u\>=0$, from \eqref{3:3} and \eqref{3:4}, we also have $$\norm{v_{r_j}}_{X_0}^2-|v_{r_j}|^{2^*}_{L^{2^*}(\Om)}\to 0.$$
Therefore, we may assume that
$$\norm{v_{r_j}}^2 \to b,\,\,\,|v_{r_j}|^{2^*}_{L^{2^*}(\Om)} \to b. $$ By Sobolev inequality,
 $\norm{v_{r_j}}^2_{X_0} \geq (|v_{r_j}|^{2^*}_{L^{2^*}(\Om)})^{2/2^*} $. 
As a result, we get $b \geq S_Kb^{2/2^*}$. We note that if $b=0,$ then we are done since that implies  $u_{r_j} \to u$  in  $X_0 $. Assume $b\not=0$. This in turn implies $b\geq S_K^\f{N}{2s}$ Then by \eqref{3:5}, we have
\be\lab{3:6}
I(u)=c-\frac{b}{2}+\frac{b}{2^*}.
\ee
 It is easy to see that $\<I'(u), u\>=0$ implies 
 \be\lab{3:7}
 I(u)=\frac{s}{N}|u|^{2^*}_{L^{2^*}(\Om)}+\bigg(\frac{1}{2}-\frac{1}{q+1}\bigg)\mu |u|^{q+1}_{L^{q+1}(\Om)}
 \ee
Combining \eqref{3:6} and \eqref{3:7} and using $q\in (0,1)$, we obtain
  \bea\lab{3:8}
  c&=&\frac{s}{N}(b+|u|^{2^*}_{L^{2^*}(\Om)})+\mu\bigg(\frac{1}{2}-\frac{1}{q+1}\bigg)|u|^{q+1}_{L^{q+1}(\Om)}\no\\
  &\geq& \frac{s}{N}(S_K^{\frac{2s}{N}}+|u|^{2^*}_{L^{2^*}(\Om)}) -\f{1-q}{2(1+q)}\mu |\Om|^{\frac{2^*-q-1}{2^*}} |u|^{q+1}_{L^{2^*}(\Om)}\no\\
 &=& \frac{s}{N}S_K^{\frac{N}{2s}}+\frac{s}{N}|u|^{2^*}_{L^{2^*}(\Om)}
  -a{\mu |u|^{q+1}_{L^{2^*}(\Om)}},
   \eea
 where $a:= \f{1-q}{2(1+q)}|\Om|^{\frac{2^*-q-1}{2^*}}>0 $. We define 
 \be\lab{3:k}
 g(t)=\frac{s}{N}t^{2^*}-a\mu t^{q+1}, \quad t\geq 0 \quad\text{and}\quad k := -\frac{1}{\mu^{\frac{2^*}{2^*-q-1}}} \min_{t\geq 0} g(t). 
 \ee
By elementary analysis it is easy to check that if  $t_0=(\f{a\mu N}{s})^{\frac{1}{2^*-q-1}} $, then $g(t)<0$ for $t \in (0,t_0)$, $g(t) \geq 0$ for $t \geq t_0$  and $g(0)=0$.
Hence, there exists $t' \in (0,t_0)$ for which $g$ attains minimum and $\min_{t>0} g(t)<0$.  Thus $k>0$. Hence from \eqref{3:8} we have 
$$c \geq \frac{s}{N}S_K^{\frac{N}{2s}} -k \mu^{\frac{2^*}{2^*-q-1}},$$
which is a contradiction to \eqref{3:c}. Therefore, $b=0 $ and the claim follows.
\hfil{$\square$}

\section{\bf Critical-concave fractional Laplace equation and sign changing solution}
In this section we consider the problem $(\mathcal{P'}_K)$ when $K(x)=|x|^{-(N+2s)}$. More precisely we study, 
 \begin{align*}
 \left(P\right)
 \begin{cases}
  \mathcal (-\De)^s u = \mu |u|^{q-1}u + |u|^{2^*-2}u=0 &\quad\mbox{in}\quad \Omega,\\
  u=0&\quad\mbox{in}\quad\mathbb{R}^n\setminus\Omega.
 \end{cases}
\end{align*}
Corresponding to $(P)$, define the  energy functional $I_{\mu}$ as follows
\bea\lab{I-mu}
I_{\mu}(u)&:=&\f{1}{2}\int_{\R^{2N}}\f{|u(x)-u(y)|^2}{|x-y|^{N+2s}}dxdy-\f{\mu}{q+1}\Iom |u|^{q+1}dx-\f{1}{2^*}\Iom |u|^{2^*}dx\no\\
&=&\f{1}{2}||u||^2-\f{\mu}{q+1}\Iom |u|^{q+1}dx-\f{1}{2^*}\Iom |u|^{2^*}dx,
\eea
where $||u||:=||u||_{X_0}$. From \cite[Lemma 9]{SerVal2}, we know
\be\lab{in:S}
S_s \displaystyle\left(\int_{\Rn} |v(x)|^{2^*}\right)^{2/2^*}\leq ||v||^2 \quad\forall\quad v\in X_0,
\ee
where
 \be\lab{S}
S_s=\inf_{v \in H^s(\Rn), v\not=0 }\frac{\displaystyle\int_{\R^{2N}}\f{|v(x)-v(y)|^2}{|x-y|^{N+2s}}dxdy}{\displaystyle\left(\int_{\Rn} |v(x)|^{2^*}\right)^{2/2^*}}.
\ee
It is known that (see \cite{CT-1}), $S_s$ is attained by $v_{\eps}\in H^s(\Rn)$, where
\be\lab{v-eps}
v_{\eps}(x):= \f{k\eps^{\frac{N-2s}{4}}}{(\eps+|x|^2)^\f{N-2s}{2}}, \quad\text{with}\quad  \eps>0, k\in\R\setminus \{0\}.
\ee
We note that $v_{\eps}\not\in X_0$. Therefore we multiply $v_{\eps}$ by a suitable cut-off function $\psi$ in order to put $v_\eps$ to $0$ outside $\Om$. For this,  fix $\de>0$. Define $\Om_1=\{x\in\Om: \text{dist}(x,\pa\Om)>\de\}$. We choose $\psi\in C^{\infty}(\Rn)$ such that $0\leq\psi\leq 1$, $\psi=1$ in $\Om_1$, $\psi=0$ in $\Rn\setminus \Om$ and $\psi>0$ in $\Om$. We define 
\be\lab{u-eps}
u_{\eps}(x) :=\psi(x)v_{\eps}(x).
\ee

To obtain sign changing solution of $(P)$, we need to study minimization problems of $I_{\mu}$ over suitable Nehari-type sets. We define the following sets in the spirit of \cite{T} (also see \cite{Chen})
\begin{align}
&\mathcal{N} := \{u \in X_0\setminus\{0\}:\<I'_{\mu}(u),u\>=0\}; \no\\
 &\mathcal{N}_0 := \Big\{u \in \mathcal{N} : (1-q)\norm{u}^2-(2^*-q-1)|u|^{2^*}_{L^{2^*}(\Om)}=0\Big\}; \notag\\
 &\mathcal{N}^+ := \Big\{u \in \mathcal{N} : (1-q)\norm{u}^2-(2^*-q-1)|u|^{2^*}_{L^{2^*}(\Om)}>0\Big\}; \notag\\
 &\mathcal{N}^- := \Big\{u \in \mathcal{N} : (1-q)\norm{u}^2-(2^*-q-1)|u|^{2^*}_{L^{2^*}(\Om)}<0\Big\}. \notag
\end{align}
From \cite{CD}, it is known that there exists $\mu_{*}>0$ such that, if $\mu \in (0,\mu_{*})$, then the following minimization problem:
\be\lab{eq:tl-al}
\tilde\al_{\mu}^+:=\inf_{u \in \mathcal{N}^+}J_{\mu}(u) \quad\text{and}\quad \tilde\al_{\mu}^-:=\inf_{u \in \mathcal{N}^-}J_{\mu}(u)\ee
achieve their minimum at $w_0$ and  $w_1$ respectively, where 
\be\lab{J-mu}
J_{\mu}(u):= \f{1}{2}||u||^2-\f{\mu}{q+1}\Iom ({u^+})^{q+1}dx-\f{1}{2^*}\Iom ({u^+})^{2^*}dx.
\ee
Moreover $w_0$ and $w_1$ are critical points of $J_{\mu}$. Using maximum principle \cite[Proposition 2.2.8]{S} and followed by a simple calculation , it can be checked  that, if $u$ is a critical point of $J_{\mu}$, then $u$ is strictly positive in $\Om$ (see \cite{BCSS}). Thus $w_0$ and $w_1$ are positive solution of $(\mathcal{P})$. Applying the Moser iteration technique it follows that any positive solution of $(P)$ is in $L^{\infty}(\Om)$ (see \cite[Proposition 2.2]{BCSS}).

%Following the same procedure as in  \cite{CD}, it can be shown that,  if $\mu \in (0,\mu_{*})$, then the following minimization problem:
%\be\lab{c-1'}
%c_0:=\inf_{u \in N^+}I_{\mu}(u) \quad\text{and}\quad c_1:=\inf_{u \in N^-}I_{\mu}(u)\ee
%achieve their minimum at $w_0$ and $w_1$ respectively. That is,
%\be\lab{c-1}
%c_0=I_{\mu}(w_0) \quad\text{and}\quad c_1=I_{\mu}(w_1).
%\ee
%Using the standard Moser iteration technique and following the proof of \cite[Proposition 2.2]{BCSS}, it can be shown that  any nontrivial solution of $(P)$ is in $L^{\infty}(\Om)$. 

We need the following lemmas in order to prove Theorem \ref{thm.2}. 
\begin{lemma}\lab{l:4i}
Suppose $w_1$ is a positive solution of $(P)$ and $u_{\eps}$ is as defined in \eqref{u-eps}. Then for every $\eps>0$, small enough
\begin{itemize}
\item[(i)] $A_1 :=\displaystyle\Iom w_1^{2^*-1}u_{\eps}dx\leq k_1\eps^\f{N-2s}{4}$;
\item[(ii)] $A_2 :=\displaystyle\Iom w_1^{q}u_{\eps}dx\leq k_2\eps^\f{N-2s}{4}$;
\item[(iii)] $A_3 :=\displaystyle\Iom w_1u_{\eps}^{q}dx\leq k_3\eps^{\f{N-2s}{4}q}$;
\item[(iv)]$A_4 :=\displaystyle\Iom w_1u_{\eps}^{2^*-1}dx\leq k_4\eps^\f{N+2s}{4}$.
\end{itemize} 
\end{lemma}
\begin{proof} Let $R,\  M>0$ be such that $\Om\subset B(0, R)$ and $|w_1|_{L^{\infty}(\Om)}<M$. Then
\Bea
(i) \quad A_1=\displaystyle\Iom w_1^{2^*-1}u_{\eps}dx &\leq& M^{2^*-1}|\psi|_{L^{\infty}(\Om)}k\eps^\f{N-2s}{4}\int_{B(0,R)} \f{dx}{(\eps+|x|^2)^\f{N-2s}{2}}\\
&\leq& C\eps^{\f{N}{2}-\f{N-2s}{4}}\int_{B(0, \f{R}{\sqrt{\eps}})}\f{dx}{(1+|x|^2)^\f{N-2s}{2}}\\
&\leq& k_1\eps^\f{N-2s}{4}.
\Eea
Proof of (ii) similar to (i).\\
\Bea
(iii) \quad A_3=\displaystyle\Iom w_1u_{\eps}^q dx &\leq& M |\psi|^{q}_{L^{\infty}(\Om)}k^q\eps^{\f{N-2s}{4}q}\int_{B(0,R)} \f{dx}{(\eps+|x|^2)^{\f{N-2s}{2}q}} \\
&\leq&  C\eps^{\f{N}{2}-\f{(N-2s)q}{4}}\int_0^ \f{R}{\sqrt{\eps}}\f{r^{N-1}dr}{(1+r^2)^\f{(N-2s)q}{2}}\\
&\leq& C\eps^{\f{N}{2}-\f{(N-2s)q}{4}}\bigg(\f{R}{\sqrt{\eps}}\bigg)^{N-(N-2s)q}\\
&\leq& k_3\eps^{\f{N-2s}{4}q}.
\Eea
(iv) can be proved as in (iii).
\end{proof}
\begin{lemma}\lab{l:4ii}
Let $u_{\eps}$ be as defined in \eqref{u-eps} and $0<q<1$. Then for every $\eps>0$, small \\

$\displaystyle\Iom |u_{\epsilon}|^{q+1}dx=\left\{\begin{array}{lll}
k_5\eps^{(\f{N-2s}{4})(q+1)} \quad &\text{if}\quad  0<q<\f{2s}{N-2s},\\
k_6\eps^\f{N}{4}|\mbox{ln}\ \eps|,  \quad &\text{if}\quad  q=\f{2s}{N-2s},\\ 
k_7\eps^{\f{N}{2}-(\f{N-2s}{4})(q+1)}  \quad &\text{if}\quad  \f{2s}{N-2s}<q<1.
\end{array}
\right.$
\end{lemma}
\begin{proof}
Choose $0<R'<R$ be such that $B(0,R')\subset\Omega_{1}\subset\Omega.$ Then $ u_{\eps}= v_{\eps}$ in $B(0,R').$ Then
\begin{align*}
 \int_{\Omega}|u_{\eps}|^{q+1}dx\geqslant \int_{B(0,R')}|u_{\eps}|^{q+1}dx= 
 k^{q+1}\eps^{\f{(N-2s)(q+1)}{4}}\int_{B(0,R')} \f{dx}{(\eps+|x|^2)^{\f{(N-2s)(q+1)}{2}}}.
\end{align*}
Proceeding as in the proof of Lemma \ref{l:4i} (iii), we have
\begin{align}\lab{4-1}
C\eps^{\f{N}{2}-\f{(N-2s)(q+1)}{4}}\int_0^ \f{R'}{\sqrt{\eps}}\f{r^{N-1}dr}{(1+r^2)^\f{(N-2s)(q+1)}{2}} &\leq
\Iom|u_{\eps}|^{q+1}dx \no\\
&\leq C\eps^{\f{N}{2}-\f{(N-2s)(q+1)}{4}}\int_0^ \f{R}{\sqrt{\eps}}\f{r^{N-1}dr}{(1+r^2)^\f{(N-2s)(q+1)}{2}}.
\end{align}  
{\it Case 1} : $0<q<\f{2s}{N-2s}$.\\
We note that
\be\lab{4-2}
\int_0^ \f{R'}{\sqrt{\eps}}\f{r^{N-1}dr}{(1+r^2)^\f{(N-2s)(q+1)}{2}}\geq C\int_{\f{R'}{2\sqrt{\eps}}}^{\f{R'}{\sqrt{\eps}}}r^{{N-1}-(N-2s)(q+1)}dr\geq \f{C}{\eps^{\f{N}{2}-\f{(N-2s)(q+1)}{2}}}
\ee
and
\be\lab{4-3}
\int_0^ \f{R}{\sqrt{\eps}}\f{r^{N-1}dr}{(1+r^2)^\f{(N-2s)(q+1)}{2}}\leq \int_{0}^{\f{R}{\sqrt{\eps}}}r^{{N-1}-(N-2s)(q+1)}dr\leq \f{C}{\eps^{\f{N}{2}-\f{(N-2s)(q+1)}{2}}}.
\ee
Substituting back \eqref{4-2} and \eqref{4-3} into \eqref{4-1}, we obtain $\displaystyle\Iom |u_{\eps}|^{q+1}dx=k_5\eps^{(\f{N-2s}{4})(q+1)}$.\\
{\it Case 2} : $q=\f{2s}{N-2s}$.\\
Then 
$$\int_0^ \f{R'}{\sqrt{\eps}}\f{r^{N-1}dr}{(1+r^2)^\f{(N-2s)(q+1)}{2}}\geq C\int_{1}^{\f{R'}{\sqrt{\eps}}}r^{{N-1}-(N-2s)(q+1)}dr\geq C'|\mbox{ln}\ \eps|.$$
\bea\int_0^ \f{R}{\sqrt{\eps}}\f{r^{N-1}dr}{(1+r^2)^\f{(N-2s)(q+1)}{2}}&\leq&\int_{0}^{1}r^{N-1}dr+\int_{1}^{\f{R}{\sqrt{\eps}}}r^{{N-1}-(N-2s)(q+1)}dr\no\\
&\leq& C(1+|\mbox{ln}\ \eps|)
\leq 2C|\mbox{ln}\ \eps|.\no\eea
Substituting back the above two expressions in \eqref{4-1}, 
we have \\$\displaystyle\Iom |u_{\eps}|^{q+1}dx =k_6\eps^\f{N}{4}|ln\ \eps|$.\\
{\it Case 3} :  $\f{2s}{N-2s}<q<1$.\\
Therefore $(N-2s)(q+1)>N$ and consequently
$$\int_0^ \f{R'}{\sqrt{\eps}}\f{r^{N-1}dr}{(1+r^2)^\f{(N-2s)(q+1)}{2}}\geq C\int_{0}^{1}r^{N-1}dr=C,$$
$$\int_0^ \f{R}{\sqrt{\eps}}\f{r^{N-1}dr}{(1+r^2)^\f{(N-2s)(q+1)}{2}}\leq\int_{0}^{1}r^{N-1}dr+\int_{1}^{\infty}r^{{N-1}-(N-2s)(q+1)}\leq C,$$
 Hence $\displaystyle\Iom |u_{\eps}|^{q+1}dx =k_7\eps^{\f{N}{2}-\f{(N-2s)(q+1)}{4}}$.
\end{proof}

\vspace{4mm}

Set 
\be\lab{tl-mu}
\tilde\mu=\bigg(\frac{1-q}{2^*-q-1}\bigg)^{\frac{1-q}{2^*-2}}\frac{2^*-2}{2^*-q-1}|\Om|^{\frac{q+1-2^*}{2^*}}S_s^{\frac{N(1-q)}{4s}+\frac{q+1}{2}}. 
\ee
Next we prove three basic lemma.
\begin{lemma}\label{l:4-iii}
Let $\mu \in (0,\tilde\mu). $ For every $u \in X_0,\  u \neq 0, $ there exists unique 
$$0<t^-(u)<t_0(u)=
\bigg(\frac{(1-q)\norm{u}^2}{(2^*-1-q)|u|^{2^*}_{L^{2^*}(\Om)}}\bigg)^{\frac{N-2s}{4s}}<t^+(u), $$ such that 
\begin{align}
&t^-(u)u \in \mathcal{N}^+ \quad\mbox{and}\quad I_{\mu}(t^-u)=\min_{t \in [0,t_0]}I_{\mu}(tu), \notag\\ 
&t^+(u)u \in \mathcal{N}^- \quad\mbox{and}\quad I_{\mu}(t^+u)=\max_{t \geq t_0}I_{\mu}(tu). \notag
\end{align}
\end{lemma}
\begin{proof}
From \eqref{I-mu},  for $t \geq 0$, 
 \be
  I_{\mu}(tu)=\frac{t^2}{2}\norm{u}^2-\frac{\mu t^{q+1}}{q+1}|u|^{q+1}_{L^{q+1}(\Om)}-\frac{t^{2^*}}{2^*}|u|^{2^*}_{L^{2^*}(\Om)}. \no
  \ee
 Therefore \be\frac{\pa}{\pa t}I_{\mu}(tu)=t^q\bigg(t^{1-q}\norm{u}^2-t^{2^*-q-1}|u|^{2^*}_{L^{2^*}(\Om)}-\mu|u|^{q+1}_{L^{q+1}(\Om)}\bigg) .\no\ee
Define 
\be\lab{eq:phi}
\phi(t)=t^{1-q}\norm{u}^2-t^{2^*-q-1}|u|^{2^*}_{L^{2^*}(\Om)}. 
\ee
By a straight forward computation, it follows that $\phi$ attains maximum at  
the point 
\be\lab{t-0}
t_0=t_0(u)=\bigg(\frac{(1-q)\norm{u}^2}{(2^*-1-q)|u|^{2^*}_{L^{2^*}(\Om)}}\bigg)^{\frac{1}{2^*-2}} .
\ee
 Thus
\be\lab{4-3'} \phi'(t_0)=0, \quad \phi'(t) >0 \quad\mbox{if}\quad t<t_0, \quad \phi'(t) <0 \quad\mbox{if}\quad t>t_0.\ee
Moreover, $\phi(t_0)=\big(\frac{1-q}{2^*-1-q}\big)^{\frac{1-q}{2^*-2}}\big(\frac{2^*-2}{2^*-1-q}\big)
\bigg(\frac{\norm{u}^{2(2^*-q-1)}}{|u|^{2^*(1-q)}_{L^{2^*}(\Om)}}\bigg)^{\frac{N-2s}{4s}}. $
Therefore using \eqref{in:S}, we have
\be\lab{4-4}
\phi(t_0) \geq \displaystyle\left(\frac{1-q}{2^*-1-q}\right)^{\frac{(1-q)(N-2s)}{4s}}\frac{2^*-2}{2^*-1-q}S_s^{\frac{N(1-q)}{4s}}\norm{u}^{q+1}. 
\ee
Using H\"older inequality followed by Sobolev inequality \eqref{in:S},  and the fact that $\mu\in (0,\tilde\mu)$, we obtain 
\bea
\mu \Iom |u|^{q+1}dx \leq \mu \norm{u}^{q+1}S^{-(q+1)/2}|\Om|^{\frac{2^*-q-1}{2^*}}&\leq& \tilde\mu \norm{u}^{q+1}S_s^{-(q+1)/2}|\Om|^{\frac{2^*-q-1}{2^*}} \no\\
&\leq& \phi(t_0),\no
\eea
where in the last inequality we have used expression of $\tilde\mu$ (see \eqref{tl-mu}) and \eqref{4-4}.
Hence, there exists $t^+(u)>t_0>t^-(u)$ such that
\be\lab{4-5}\phi(t^+)=\mu\Iom|u|^{q+1}=\phi(t^-)\quad\mbox{and}\quad \phi'(t^+)<0<\phi'(t^-). \ee
\includegraphics[scale=0.25]{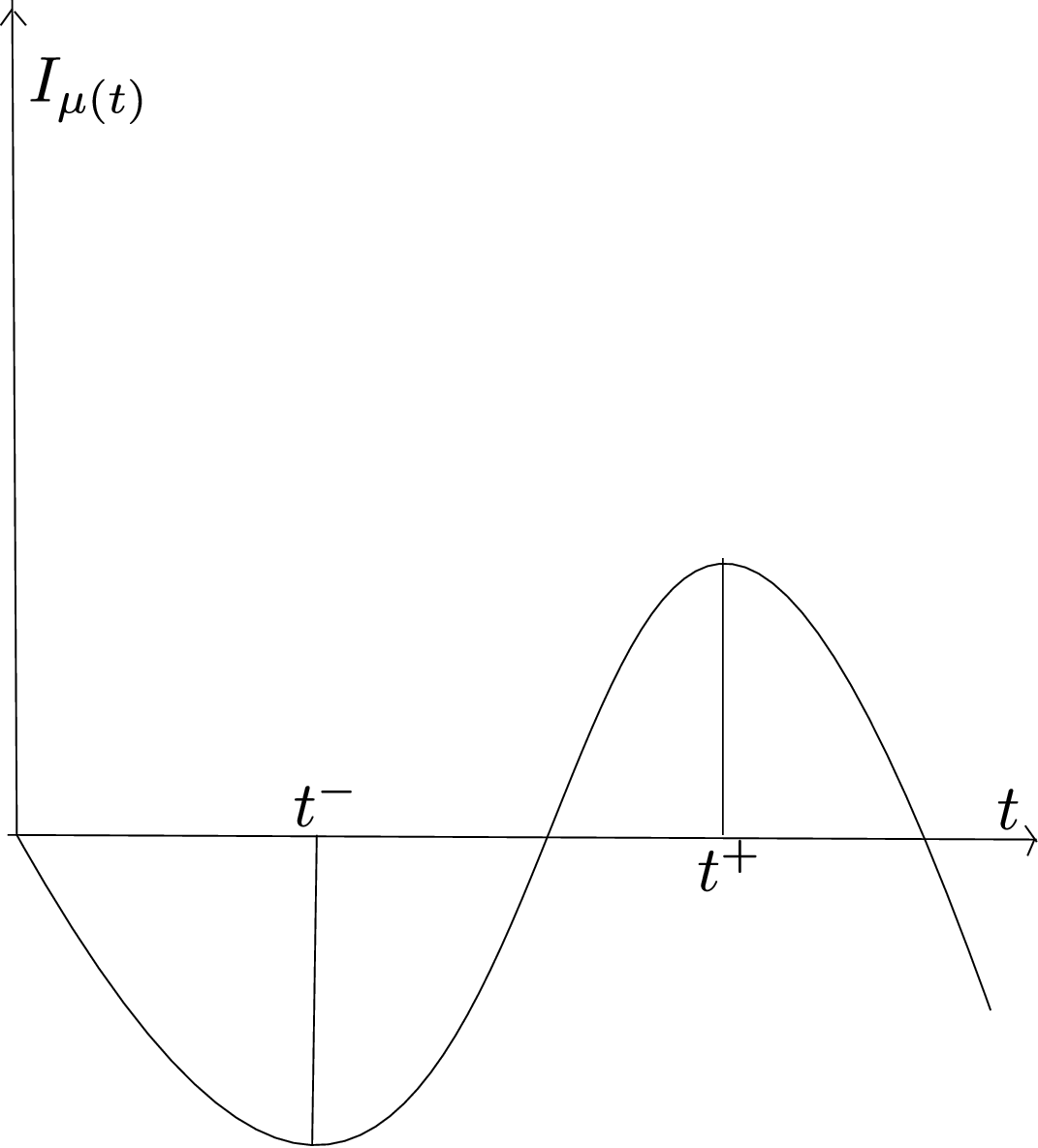} \hspace{2cm}
\includegraphics[scale=0.25]{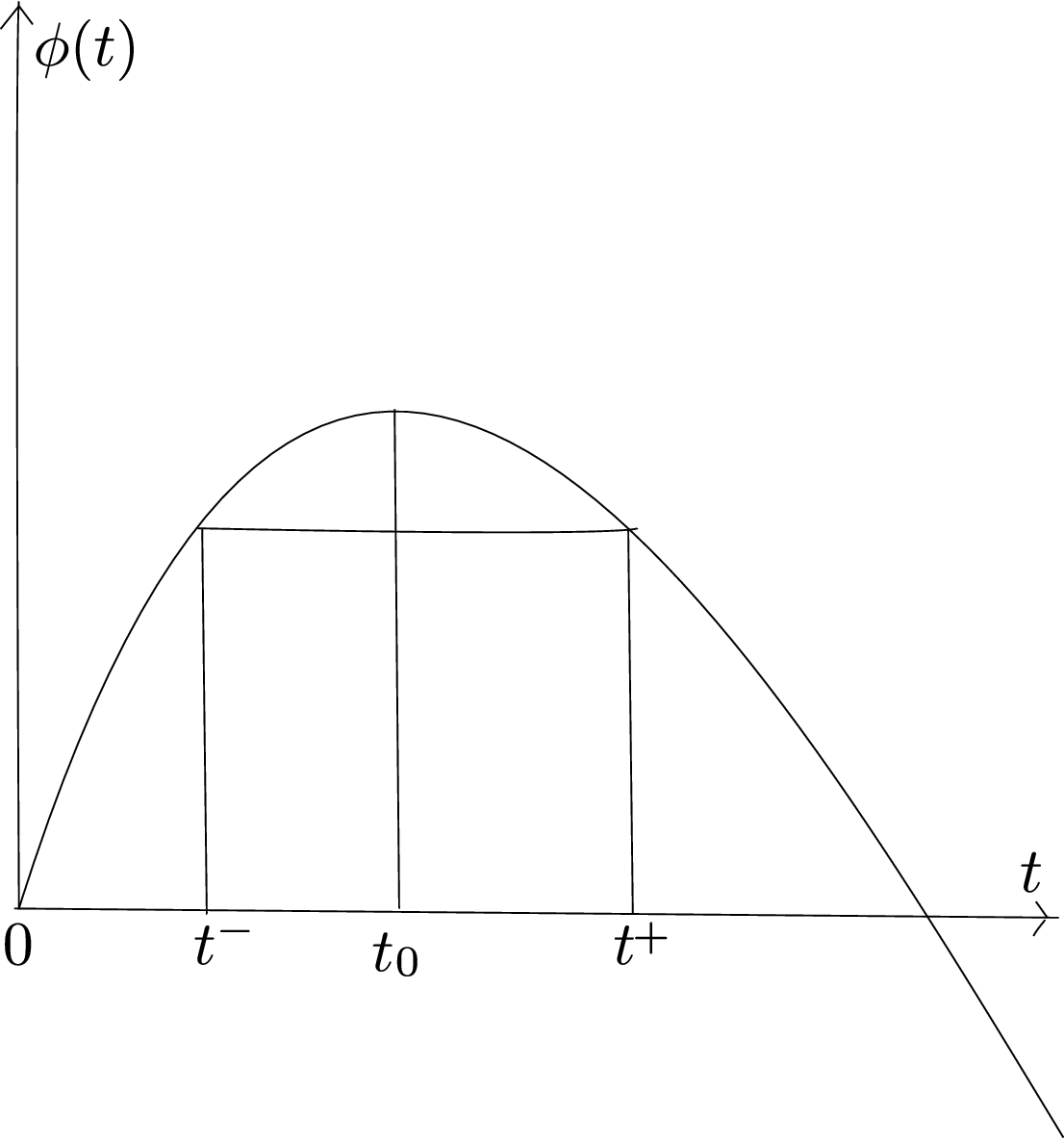}

This in turn, implies $t^+u \in \mathcal{N}^-$ and  $t^-u \in \mathcal{N}^+ $. Moreover, using \eqref{4-3'} and \eqref{4-5} 
in the expression of $\frac{\pa}{\pa t}I_{\mu}(tu)$, we have
$$\quad \frac{\pa}{\pa t}I_{\mu}(tu) > 0 \quad\mbox{when}\quad t\in(t^-,t^+) \quad\text{and}$$
$$\frac{\pa}{\pa t}I_{\mu}(tu) < 0 \quad \mbox{when} \quad t\in[0,t^-)\cup(t^+,\infty),$$
$$\frac{\pa}{\pa t}I_{\mu}(tu)=0 \quad\text{when}\quad t=t^{\pm}.$$
We note that $I_{\mu}(tu)=0$ at $t=0$ and strictly negative when $t>0$ is small enough. Therefore it is easy to conclude that 
$$\max_{t \geq t_0}I_{\mu}(tu)=I_{\mu}(t^+u) \quad\mbox{and} \quad \min_{t \in [0,t_0]}I_{\mu}(tu)=I_{\mu}(t^-u).$$
\end{proof}

\begin{lemma}\label{l:4-iv}
Let $\tilde\mu$ be  defined as in $\eqref{tl-mu}$.  Then $\mu \in (0,\tilde\mu), $ implies $\mathcal{N}_0=\emptyset$. 
\end{lemma}
\begin{proof}
 Suppose not. Then there exists $ w \in \mathcal{N}_0$ such that $w\not=0$ and
 \begin{align}\label{4-6}
(1-q)\norm{w}^2-(2^*-q-1)|w|^{2^*}_{L^{2^*}(\Om)}=0.
\end{align}
The above expression combined with Sobolev inequality \eqref{in:S} yields
\begin{align}\label{4-7}
\norm{w} \geq S_s^{N/4s}\displaystyle\left(\frac{1-q}{2^*-1-q}\right)^{\frac{N-2s}{4s}}.
\end{align}
As $w \in \mathcal{N}_0 \subseteq \mathcal{N}$, using \eqref{4-6} and H\"older inequality followed by Sobolev inequality, we get
\Bea
 0&=&\norm{w}^2-|w|^{2^*}_{L^{2^*}(\Om)}-\mu|w|^{q+1}_{L^{q+1}(\Om)}\\
 &\geq&\norm{w}^2-\displaystyle\left(\f{1-q}{2^*-q-1}\right)||w||^2-\mu|\Om|^{1-\frac{q+1}{2^*}}S_s^{-(q+1)/2}||w||^{q+1}.
\Eea 
Combining the above inequality with \eqref{4-7} and using $\mu<\tilde\mu$, we have
\bea 0&\geq&\norm{w}^{q+1}\displaystyle\left[\bigg(\frac{2^*-2}{2^*-q-1}\bigg)\bigg(\frac{1-q}{2^*-q-1}\bigg)^{\frac{(N-2s)(1-q)}{4s}}S_s^{\frac{N(1-q)}{4s}}-
 \mu|\Om|^{1-\frac{q+1}{2^*}}S_s^{-(q+1)/2}\right]\no\\
 &>& 0,\no\eea
which is a contradiction. This completes the proof.
\end{proof}
\begin{lemma}\label{l:4-v}
Let $\tilde\mu$ be defined as in $\eqref{tl-mu}$ and $\mu \in (0,\tilde\mu)$. Given $u \in \mathcal{N}, $ there exists $\rho_u>0$ and a differential function $g_{\rho_u}:B_{\rho_u}(0) \to \R^+$
satisfying the following:
\begin{align}
 &g_{\rho_u}(0)=1, \notag\\
 &\big(g_{\rho_u}(w)\big)(u+w) \in \mathcal{N}\quad \forall\quad w \in B_{\rho_u}(0), \notag\\
 &\<g'_{\rho_u}(0),\phi\>=\frac{2\<u,\phi\>-2^*\Iom |u|^{2^*-2}u\phi-(q+1)\mu\Iom |u|^{q-1}u \phi}{(1-q)\norm{u}^2-
       (2^*-q-1)|u|^{2^*}_{L^{2^*}(\Om)}}. \notag
\end{align}
\end{lemma}
\begin{proof}
 Define $F:\R^{+} \times X_0 \to \R$ as follows:
 $$F(t,w)=t^{1-q}\norm{u+w}^2-t^{2^*-q-1}|u+w|^{2^*}_{L^{2^*}(\Om)}-\mu |u+w|^{q+1}_{L^{q+1}(\Om)}. $$
We note that $u \in \mathcal{N}$ implies
 $$F(1,0)=0,  \quad\text{and}\quad \frac{\pa F}{\pa t}(1,0)=(1-q)\norm{u}^2-(2^*-q-1)|u|^{2^*}_{L^{2^*}(\Om)}\not=0. $$
Therefore, by Implicit function theorem, there exists neighbourhood $B_{\rho_u}(0)$ for some $\rho_u>0$ and a $C^1$ function
 $g_{\rho_u}:B_{\rho_u}(0) \to \R^+$ such that
 \begin{align}
  &(i)\ g_{\rho_u}(0)=1, \quad (ii)\ F(g_{\rho_u}(w),w)=0, \,\, \forall\  w \in B_{\rho_u}(0), \no\\ 
  &(iii) F_t(g_{\rho_u}(w),w)\not=0, \,\, \forall\  w \in B_{\rho_u}(0), \quad
  (iv)\  \<g'_{\rho_u}(0),\phi\>=-\frac{\<\frac{\pa F}{\pa w}(1,0),\phi\>}{\frac{\pa F}{\pa t}(1,0)}. \notag
 \end{align}
 Multiplying (ii) by $(g_{\rho_u}(w))^{q+1}$, it follows that $(g_{\rho_u}(w))(u+w)\in \mathcal{N}$. In fact, simplifying (iii), we obtain
 $$(1-q)(g_{\rho_u}(w))^2||u+w||^2-(2^*-q-1)(g_{\rho_u}(w))^{2^*}|u+w|_{2^*}^{2^*}\not=0 \quad\forall\  w \in B_{\rho_u}(0).$$ Thus $\big(g_{\rho_u}(w)\big)(u+w)\in \mathcal{N}^{-}\cup \mathcal{N}^{+}$ for every $  w \in B_{\rho_u}(0)$. The last assertion of the lemma follows from (iv).
 \end{proof}
\subsection{Sign changing critical points of $I_{\mu}$}

Define $$\mathcal{N}^{-}_1:=\{u\in \mathcal{N}: u^{+}\in \mathcal{N}^{-}\},$$
$$\mathcal{N}^{-}_2:=\{u\in \mathcal{N}: -u^{-}\in \mathcal{N}^{-}\},$$
We set
\begin{align}\label{ba-1-2}
 \ba_1=\inf_{u \in \mathcal{N}_1^-}I_\mu(u) \quad\text{and}\quad \ba_2=\inf_{u \in \mathcal{N}_2^-}I_\mu(u).
\end{align}

\begin{theorem} \label{t:4ii}
Assume $0<\mu<\text{min}\{\tilde\mu, \mu_{*}, \mu_1\}$, where $\mu_1$ is as in Lemma \ref{l:6-ii}, $\tilde\mu$ is as in \eqref{tl-mu} and $\mu_*$ is chosen such that $\tilde\al_{\mu}^-$ is achieved in $(0, \mu_*)$.  Let $\ba_1$, $\ba_2$, $\tilde\al_{\mu}^-$ be  defined as in \eqref{ba-1-2} and \eqref{eq:tl-al} respectively. 
\begin{itemize}
\item[(i)] Let $\ba_1<\tilde\al_{\mu}^{-}$. Then there exists a sign changing critical point $\tilde w_1$ of  $I_\mu$ such that  $\tilde w_1\in \mathcal{N}_1^-$ and $I_\mu(\tilde w_1)=\ba_1$.  
\item[(ii)]
If $\ba_2< \tilde\al_{\mu}^{-}$, then there exists a sign changing critical point $\tilde w_2$ of $I_\mu$ such that   in $\tilde w_2\in \mathcal{N}_1^-$ and $I_\mu(\tilde w_2)=\ba_2$.
 \end{itemize}
\end{theorem}
\begin{proof}
(i) Let $\ba_1< \tilde\al_{\mu}^{-}$.

\vspace{2mm}

{\bf Claim 1:}  $\mathcal{N}_1^-$ and $\mathcal{N}_2^-$ are closed sets. \\
To see this, let $\{u_n\} \subset \mathcal{N}_1^-$ such that $u_n \to u$ in $X_0$. It is easy to note that $|u_n|, |u|\in X_0$ and  $|u_n| \to |u|$ in $X_0$. This in turn implies 
$u_n^+ \to u^+$ in $X_0$ and  $L^\ga(\Rn)$ for $\ga \in [1,2^*]$ (by \eqref{in:S}). Since, $u_n \in \mathcal{N}_1^-$, we have  $u_n^+ \in \mathcal{N}^-$. Therefore
 \be\label{4-10}
  \norm{u_n^+}^2-|u_n^+|^{2^*}_{L^{2^*}(\Om)}-\mu|u_n^+|^{q+1}_{L^{q+1}(\Om)}=0
  \ee
  and
\be \label{4-11} 
  (1-q)\norm{u_n^+}^2-(2^*-q-1)|u_n^+|^{2^*}_{L^{2^*}(\Om)}<0 \,\,\forall\  n \geq 1. 
 \ee
Passing to the limit as $n\to\infty$, we obtain $u^+\in \mathcal{N}$ and $(1-q)\norm{u^+}^2-(2^*-q-1)|u^+|^{2^*}_{L^{2^*}(\Om)}\leq 0$. But, from Lemma \ref{l:4-iv}, we know $\mathcal{N}_0=\emptyset$. Therefore $u^+\in \mathcal{N}^{-}$ and hence $\mathcal{N}_1^{-}$ is closed. Similarly it can be shown that $\mathcal{N}_2^{-}$ is also closed. Hence claim 1 follows.
%Therefore $(1-q)\norm{u^+}^2-(2^*-q-1)|u^+|^{2^*}_{L^{2^*}(\Om)}$  can be $0$ only if $u^{+}=0$. On the other hand, we note that if $u^{+}=0$ then we get a contradiction passing the limit in \eqref{4-11} (since $u_n^+ \to u^+$ in $X_0$ and  $L^\ga(\Rn)$). Therefore $u^{+}\not= 0$  and this implies 

\vspace{2mm}

By Ekeland Variational Principle there exists sequence $\{u_n\} \subset \mathcal{N}_1^-$ 
such that 
\be\lab{4-11'}
I_\mu(u_n) \to \ba_1 \quad\text{and}\quad I_\mu(z) \geq I_\mu(u_n)-\frac{1}{n}\norm{u_n-z} \quad\forall \quad z\in \mathcal{N}_1^-.
\ee

\vspace{2mm}

{\bf Claim 2:}  $\{u_n\}$ is uniformly bounded in $X_0$. \\
To see this, we notice $u_n \in \mathcal{N}_1^-$ implies $u_n \in \mathcal{N}$ and this in turn implies $\<I'_\mu(u_n),u_n\>=0$, that is,
$$\norm{u_n}^2=|u_n|^{2^*}_{L^{2^*}(\Om)}+|u_n|^{q+1}_{L^{q+1}(\Om)}. $$
Since $I_\mu(u_n)\to\ba_1$, using the above equality in the expression of $I_\mu(u_n)$, we get, for $n$ large enough
\Bea
\frac{s}{N}\norm{u_n}^2&\leq& \ba_1+1+\displaystyle\left(\frac{1}{q+1}-\frac{1}{2^*}\right)|u_n|^{q+1}_{L^{q+1}(\Om)}\\
&\leq& C(1 + \norm{u_n}^{q+1}).
\Eea
 This implies $\{u_n\}$ is uniformly bounded in $X_0$.

\vspace{2mm}

{\bf Claim 3:}  There exists $b >0$ such that $\norm{u_n^-} \geq b$ for all $n \geq 1.$ \\
Suppose the claim is not true.  Then for each $k \geq 1,$ there exists $u_{n_k}$ such that
\begin{align}\label{4-12}
\norm{u_{n_k}^-}<\frac{1}{k}\,\,\, \forall\ k \ \geq 1.
\end{align}
We note that for any $u\in X_0$, we have
\bea\lab{4-13}
\norm{u}^2&=&\int_{\R^{2N}}\frac{|u(x)-u(y)|^2}{|x-y|^{N+2s}}dxdy\no\\
&=&\int_{\R^{2N}}\frac{|(u^+(x)-u^+(y))-(u^-(x)-u^-(y))|^2}{|x-y|^{N+2s}}dxdy\no\\
&=& \norm{u^+}^2+\norm{u^-}^2+2\int_{\R^{2N}}\frac{u^+(x)u^-(y)+u^+(y)u^-(x)}{|x-y|^{N+2s}}dxdy\no\\
&\geq& \norm{u^+}^2+\norm{u^-}^2
\eea
By a simple calculation, it follows 
\be\lab{4-14}|u|^{2^*}_{L^{2^*}(\Om)}=|u^+|^{2^*}_{L^{2^*}(\Om)}+|u^-|^{2^*}_{L^{2^*}(\Om)}\quad\text{and}\quad
|u|^{q+1}_{L^{q+1}(\Om)}=|u^+|^{q+1}_{L^{q+1}(\Om)}+|u^-|^{q+1}_{L^{q+1}(\Om)}. \ee
Combining \eqref{4-13} and \eqref{4-14}, we obtain
\be\lab{4-14'}
I_\mu(u) \geq I_\mu(u^+)+I_\mu(u^-) \quad\forall\quad u\in X_0.
\ee 
Moreover, $\eqref{4-12}$ implies 
$\norm{u_{n_k}^-} \to 0$ and therefore by Sobolev inequality 
$$|u_{n_k}^-|_{L^{2^*}(\Om)} \to 0,\,\,\,|u_{n_k}^-|_{L^{q+1}(\Om)} \to 0, \quad\text{as}\quad k \to \infty.$$
Consequently, $I_\mu(u_{n_k}^-) \to 0$ as $k \to \infty$.  As a result, we have
$$
\ba_1=I_\mu(u_{n_k})+o(1)\geq I_\mu(u^+_{n_k})+I_\mu(u^-_{n_k})+o(1)=J_{\mu}(u^+_{n_k})+o(1)\geq \tilde\al_{\mu}^{-}+o(1).
$$
This is a contradiction to the hypothesis. Hence  claim 3 follows.

\vspace{2mm}

{\bf Claim 4}: $I'_\mu(u_n) \to 0$ in $(X_0)'$ as $n \to \infty$.

Since $u_n\in \mathcal{N}^{-}_1\subset \mathcal{N}$, By Lemma \ref{l:4-v} applied to the element $u_{n}$, there exists 
\be\lab{g_n}
\rho_n :=\rho_{u_n} \quad\text{and}\quad g_n:=g_{\rho_{u_n}}
\ee such that 
\be\lab{4-15}
g_n(0)=1, \quad \big(g_n(w)\big)(u_n+w) \in N \quad\forall\quad w \in B_{\rho_n}(0).
\ee  Choose $0<\tilde\rho_n<\rho_n$ such that $\tilde\rho_n \to 0$.  Let $v\in X_0$ with $||v||=1$.

Define $$v_n := -\tilde\rho_n [v^+\chi_{\{u_n \geq 0\}}-v^-\chi_{\{u_n \leq 0\}}]$$
and 
\Bea
z_{\tilde\rho_n}&:=&\big(g_n(v_n^-)\big)(u_n-v_n)\\
&=:&z_{\tilde\rho_n}^1-z_{\tilde\rho_n}^2,
\Eea
where $z_{\tilde\rho_n}^1:=\big(g_n(v_n^-)\big)(u_n^{+} +\tilde\rho_n v^+\chi_{\{u_n \geq 0\}})$ and 
$z_{\tilde\rho_n}^2:=\big(g_n(v_n^-)\big)(u_n^{-} +\tilde\rho_n v^-\chi_{\{u_n \leq 0\}}).$
Note that $v_n^-=\tilde\rho_n v^+\chi_{\{u_n \geq 0\}}.$ So, $||v_n^-|| \leq \tilde\rho_n ||v||\leq \tilde\rho_n.$
Hence taking $w=v_n^-$ in \eqref{4-15}
 we have, $z_{\tilde\rho_n}^+=  z_{\tilde\rho_n}^1 \in \mathcal{N}^{-}$ so $z_{\tilde\rho_n} \in \mathcal{N}_{1}^-.$
%We define
%$z_{\tilde\rho_n}: =\big(g_n(-v_n)\big)(u_n-v_n)$. Clearly, $z_{\tilde\rho_n}\in N_{\mu}^-$.
Hence, 
$$I_\mu(z_{\tilde\rho_n}) \geq I_\mu(u_n)-\frac{1}{n} ||u_n-z_{\tilde\rho_n}||. $$

This implies,
\bea\label{22.02.E1}
\frac{1}{n}||u_n-z_{\tilde\rho_n}||&\geq& I_\mu(u_n)-I_\mu(z_{\tilde\rho_n})\notag\\
&=&\<I'_\mu(u_n),u_n-z_{\tilde\rho_n}\>+o(1) ||u_n-z_{\tilde\rho_n}||\notag\\
&=&-\<I'_\mu(u_n),z_{\tilde\rho_n}\>+o(1) ||u_n-z_{\tilde\rho_n}||,
\eea
as $\<I'_\mu(u_n),u_n\>=0 $ for all $n.$
Let $w_n=\tilde\rho_n v.$ Then,
\begin{align}\label{D1}
\frac{1}{n} ||u_n-z_{\tilde\rho_n}|| \geq -\<I'_\mu(u_n),w_n+z_{\tilde\rho_n}\>+\<I'_\mu(u_n),w_n\>\notag\\
             +o(1) ||u_n-z_{\tilde\rho_n}||.
\end{align}
Now, $\<I'_\mu(u_n),w_n\>=\<I'_\mu(u_n),\tilde\rho_n v\>=\tilde\rho_n\<I'_\mu(u_n), v\>.$
Define $$\overline{v_n}:=v^+\chi_{\{u_n \geq 0\}}-v^-\chi_{\{u_n \leq 0\}}.$$

So, $z_{\tilde\rho_n}=g_n(v_n^-)(u_n-\tilde\rho_n\overline{v_n}).$
Hence we have,
\begin{align}\label{D2}
\<I'_\mu(u_n), w_n+z_{\tilde\rho_n}\>= \<I'_\mu(u_n),w_n+g_n(v_n^-)(u_n-\tilde\rho_n\overline{v_n})\> =\<I'_\mu(u_n),\tilde\rho_n v-g_n(v_n^-)\tilde\rho_n\overline{v_n}\>\no\\
 =\tilde\rho_n\<I'_\mu(u_n),v-g_n(v_n^-)\overline{v_n}\>
\end{align}

Using \eqref{D2} in \eqref{D1}, we have
\begin{align}\label{22.02.E2}
\frac{1}{n} ||u_n-z_{\tilde\rho_n}|| \geq -\tilde\rho_n\<I'_\mu(u_n),v-g_n(v_n^-)\overline{v_n}\>\notag\\
             +\tilde\rho_n\<I'_\mu(u_n),v\>+o(1) ||u_n-z_{\tilde\rho_n}||.
\end{align}

\iffalse+++++++++++++
\textit{subclaim}:
\begin{align}\label{D3}
\<I'_\mu(u_n),v-g_n(v_n^-)\overline{v_n}\>=o(1).
\end{align}
To see this,
++++++++++\fi
First we will estimate $\<I'_\mu(u_n),v-g_n(v_n^-)\overline{v_n}\>$. For this,
\Bea
v-g_n(v_n^-)\overline{v_n}&=&v^+-v^--g_n(v_n^-)[v^+\chi_{\{u_n \geq 0\}}-v^-\chi_{\{u_n \leq 0\}}]\\
&=&v^+[g_n(0)-g_n(v_n^-)\chi_{\{u_n \geq 0\}}]-v^-[g_n(0)-g_n(v_n^-)\chi_{\{u_n \leq 0\}}]\\
&=&- v^+[\<g_n'(0),v_n^-\>+o(1) ||v_n^-||]+v^-[\<g_n'(0),v_n^-\>+o(1)||{v_n^-}||]\\
&=&- v^+\tilde\rho_n[\<g_n'(0),v^+\>+o(1)||{v^+}||]+v^-\tilde\rho_n[\<g_n'(0),v^+\>+o(1) ||{v^+}||]\\
&=&-\tilde\rho_n\big[\<g_n'(0),v^+\>+o(1) ||{v^+}||\big]v.
\Eea
Therefore,
\begin{align}\lab{Mar-5-2}
 \<I'_\mu(u_n),v-g_n(v_n^-)\overline{v_n}\> 
     =-\tilde\rho_n\big(\<g_n'(0),v^+\>+o(1)\norm{v^+}\big)\<I'_\mu(u_n),v\>.
\end{align}
%since $|\<g_n'(0),v^+\>|$ is uniformly bounded (see Lemma \ref{l:6-ii} in  Appendix) and 
%{\color{red}  $\<I'_\mu(u_n),v^+\>$, $\<I'_\mu(u_n),v^-\>$ are uniformly bounded . (this you need to check!!) }

{\bf Claim :}  $g_n(v_n^-)$ is uniformly bounded in $X_0.$ 

To see this, we observe that  from \eqref{4-15} we have, $g_n(v_n^-)(u_n^++v_n^-) \in \mathcal{N}^{-} \subset \mathcal{N},$ which implies,
$$||{c_n\tilde\psi_n}||^2-\mu|c_n\tilde\psi_n|^{q+1}_{L^{q+1}(\Om)}-|c_n\tilde\psi_n|_{L^{2^*}(\Om)}^{2^*}=0,$$
where $c_n:=g_n(v_n^-)$ and $\tilde\psi_n:=u_n^+ + v_n^-.$
Dividing by $c_n^{2^*}$ we have,
\begin{align}\label{D4}
 c_n^{2-2^*}||{\tilde\psi_n}||^2-\mu c_n^{q+1-2^*}|\tilde\psi_n|^{q+1}_{L^{q+1}(\Om)}=|\tilde\psi_n|_{L^{2^*}(\Om)}^{2^*}.
\end{align}
Note that $||\tilde\psi_n||$ is uniformly bounded above as $||u_n||_{X_0}$ is uniformly bounded and $\tilde\rho_n=o(1)$. Also,
$||{\tilde\psi_n}|| \geq ||{u_n^{+}}||-\tilde\rho_n ||{v}||$. Note that $||{u_n^{+}}|| \geq \tilde b$ for large $n$.
If not, then $||{u_n^+}|| \to 0$ as $n \to \infty.$ As $u_n \in \mathcal{N}_{1}^-,$ so $u_n^+ \in \mathcal{N}^-.$
Now, $\mathcal{N}^{-}$ is a closed set and $0 \notin \mathcal{N}^{-}$ and therefore $||{u_n^{+}}|| \not\to 0$ as $n \to \infty.$  
Thus there exists $\tilde b \geq 0$ such that $||{u_n^+}||\geq \tilde b>0$.  This in turn implies that 
$||\tilde\psi_n|| \geq C$, for some $C>0$ by choosing $\tilde\rho_n$ small enough. Consequently, if $c_n$ is not uniformly bounded,
we obtain LHS of \eqref{D4} converges to $0$ as $n\to\infty $. 

On the other hand,  $$|\tilde\psi_n|_{L^{2^*}(\Om)} \geq|u_n^{+}|_{L^{2^*}(\Om)}-\tilde\rho_n|v|_{L^{2^*}(\Om)}>c,$$ for some positive constant $c$ as $\rho_n=o(1)$ and $u_n^+\in \mathcal{N}^-$ implies 
$$(2^*-1-q)|u_n^+|_{L^{2^*}(\Om)}^{2^*}>(1-q)||u_n^+||^2>(1-q)\tilde b^2.$$ 
Hence, the claim follows.

\vspace{3mm}

Now using the fact that $g_n(0)=1$ and the above claim we obtain
\Bea
||{u_n-z_{\tilde\rho_n}}||&\leq& ||{u_n}||\big|1-g_n(v_n^-)\big|+\tilde\rho_n ||{\overline{v_n}}||g_n(v_n^-)\\
&\leq& ||{u_n}||\big[|\<g_n'(0),v_n^-\>|+o(1) ||{\overline{v_n}||}\big]+\tilde\rho_n||v||g_n(v_n^-)\\
&\leq& \tilde\rho_n \big[||{u_n}||\<g_n'(0),\overline{v_n}^+\>+o(1)||{v}||+||{v}||g_n(v_n^-)\big]\\
&\leq& \tilde\rho_n C.
\Eea
Substituting this and \eqref{Mar-5-2} in \eqref{22.02.E2} yields
$$\tilde\rho_n\bigg(\<g_n'(0),v^+\>+o(1) ||{v^+}||\bigg)\<I'_\mu(u_n),v\>+  \<I'_\mu(u_n),v\>\tilde\rho_n+\tilde\rho_n o(1) \leq \tilde\rho_n.\f{C}{n}.$$
This implies
$$\bigg[\big(\<g_n'(0),v^+\>+o(1) ||{v^+}||\big)+1\bigg]    \<I'_\mu(u_n),v\> \leq \f{C}{n}+o(1)\quad\mbox{for all}\quad n \geq n_0.$$ Since $|\<g_n'(0),v^+\>|$ is uniformly bounded (see Lemma \ref{l:6-ii} in Appendix) , letting $n \to \infty$ we have
$I'_\mu(u_n) \to 0$ in $(X_0)'.$ Hence the step 4 follows.

\vspace{2mm}

Therefore, $\{u_n\}$ is a (PS) sequence of $I_{\mu}$ at level $\ba_1< \tilde\al_{\mu}^{-} $. From \cite [Proposition 4.2]{CD}, it follows that 
$$\tilde\al_{\mu}^{-}<\frac{s}{N}S_s^{\frac{N}{2s}}-M\mu^{\frac{2^*}{2^*-q-1}}\quad\mbox{for}\quad \mu \in(0,\mu_*),  $$ 
% \quad (!! \ Is \ it\  not\ \tilde\al_{\mu}^-<..),$$ 
%$$\tilde\al_{\mu}^{-}:=\inf_{N^-}J_{\mu}(u), \quad J_{\mu}(u):= \f{1}{2}||u||^2-\f{\mu}{q+1}\Iom ({u^+})^{q+1}dx-\f{1}{2^*}\Iom ({u^+})^{2^*}dx$$
where  $M=\f{\big(2N-(N-2s)(q+1)\big)(1-q)}{4(q+1)}\big(\f{(1-q)(N-2s)}{4s}\big)^\f{q+1}{2^*-q-1}|\Om|.$ 
%It is also known from \cite{CD} that $\tilde\al_{\mu}^{-}$ is achieved by some function (say) $\bar w_1$ and $\bar w_1\in N^-$ is a positive solution of $(\mathcal{P})$. 

Therefore,

$$\ba_1<\tilde\al_{\mu}^{-}<\frac{s}{N}S_s^{\frac{N}{2s}}-M\mu^{\frac{2^*}{2^*-q-1}}.$$

On the other hand, it follows from the proof of Theorem \ref{thm.1} (see \eqref{3:c})that $I_\mu$ satisfies PS at level $c$ for
$$c<\frac{s}{N}S_s^{\frac{N}{2s}}-k\mu^{\frac{2^*}{2^*-q-1}},$$
where $k$ is as in \eqref{3:k}. By elementary analysis, it follows $k=M$.
Therefore there exists $u\in X_0$ such that $u_n\to u$ in $X_0$. 
By doing a simple calculation we get $u_n^- \to u^-$  in $X_0$. 
Consequently, by Claim 3  $\norm{u^-} \geq b$. As $\mathcal{N}_1^-$ is a closed set and $u_n \to u$, we obtain $u \in \mathcal{N}_1^-$,  that is, $u^+ \in \mathcal{N}^-$ and $u^+ \neq 0. $
Therefore $u$ is a solution of $(\mathcal{P})$ with $u^+$ and $u^-$ are both nonzero. Hence, $u$ is a sign-changing solution of $(\mathcal{P})$. Define $\tilde w_1:=u$. This completes the proof of part (i) of the theorem.

\vspace{2mm}

Proof of part (ii) is similar to part (i) and we omit the proof. 
\end{proof}

\begin{theorem}\lab{t:4i}
Let $\ba_1$, $\ba_2$, $ \tilde\al_{\mu}^{-}$ be defined as in \eqref{ba-1-2} and \eqref{eq:tl-al} respectively. Assume $\ba_1, \ba_2\geq  \tilde\al_{\mu}^{-}$. Then there exists $\mu_0>0$ such that for any $\mu\in(0,\mu_0)$, $I_{\mu}$ has a sign changing critical point. 
\end{theorem}

We need the following Proposition to prove the above Theorem \ref{t:4i}.
\begin{proposition}\lab{p:limit}
Let $N>6s$ and $\f{1}{2}(\f{N+2s}{N-2s})<q<1$. Assume $0<\mu<\text{min}\{\mu_{*}, \tilde\mu\}$, where  $\tilde\mu$ is as defined in $\eqref{tl-mu}$ and $\mu_{*}>0$ is chosen such that  $\tilde\al_\mu^{-}$  is achieved in $(0, \mu_{*})$. Then for $\eps>0$ sufficiently small , we have $$\sup_{a\geq 0,\ b\in\R} I_{\mu}(a w_1-bu_{\eps})< \tilde\al_{\mu}^{-}+\f{s}{N}S_s^\f{N}{2s},$$
where $w_1$ and $u_{\eps}$ are as in \eqref{eq:tl-al} and \eqref{u-eps} respectively.
\end{proposition}
To prove the above proposition, we need the following lemmas.
\begin{lemma}\lab{l:4vi}
Let  $w_1$ and $\mu$ be as in Proposition \ref{p:limit}. Then $$\sup_{s>0}I_{\mu}(s w_1)= \tilde\al_{\mu}^{-}.$$
\end{lemma} 
\begin{proof}
By the definition of   $\tilde\al_{\mu}^{-}$, we have  $\tilde\al_{\mu}^{-}=\inf_{u \in \mathcal{N}^-}J_{\mu}(u)= J_{\mu}(w_1)=I_{\mu}(w_1)$. In the last equality we have used the fact that $w_1>0$.  Define $g(s):=I_{\mu}(sw_1)$. From the proof of Lemma \ref{l:4-iii}, it follows that there exists only two critical points of $g$, namely $t^{+}(w_1)$ and $t^{-}(w_1)$ and $\max_{s>0}g(s)=g(t^{+}(w_1))$. On the other hand $\<{I'}_{\mu}(w_1), v \>=0$ for every $v\in X_0$. Therefore $g'(1)=0$. This in turn implies either $t^{+}(w_1)=1$ or $t^{-}(w_1)=1$. \\
{\it Claim}: $t^{-}(w_1)\not=1$.\\
To see this, we note that $t^{-}(w_1)=1$ implies $t^{-}(w_1)w_1\in \mathcal{N}^{-}$ as $w_1\in \mathcal{N}^{-}$. Also, from Lemma \ref{l:4-iii}, we know $t^{-}(w_1)w_1\in \mathcal{N}^{+}$. Thus $\mathcal{N}^{+}\cap \mathcal{N}^{-}\not=\emptyset$, which is a contradiction. Hence the claim follows.\\
Therefore $t^{+}(w_1)=1$ and this completes the proof.
\end{proof}
\begin{lemma}\lab{l:4vii}
Let $u_{\eps}$ be as in \eqref{u-eps} and  $\mu$ be as in Proposition \ref{p:limit}. Then for $\eps>0$ sufficiently small, we have $$\sup_{t\in\R}I_{\mu}(t u_{\eps})=\f{s}{N}S_s^\f{N}{2s}+ C\eps^{\f{(N-2s)N}{2s}} - k_8|u_{\eps}|^{q+1}.$$
\end{lemma}
\begin{proof}
Define $\tilde\phi(t)=\f{t^2}{2}||u_{\eps}||^2-\f{t^{2^*}}{2^*}|u_{\eps}|_{L^{2^*}(\Om)}^{2^*}$. Thus $I_{\mu}(t u_{\eps})=\tilde\phi(t)-\mu\f{t^{q+1}}{q+1}|u_{\eps}|^{q+1}_{L^{q+1}(\Om)}$. On the other hand, applying the analysis done in Lemma \ref{l:4-iii} to $u_{\eps}$, we obtain there exists $(t_0)_{\eps}= \bigg(\frac{(1-q)\norm{u_{\eps}}^2}{(2^*-1-q)|u_{\eps}|^{2^*}_{L^{2^*}(\Om)}}\bigg)^{\frac{N-2s}{4s}}
<t^{+}_{\eps}$ such that 
\Bea
\sup_{t\in\R}I_{\mu}(t u_{\eps})= \sup_{t\geq 0}I_{\mu}(t u_{\eps})=I_{\mu}(t^{+}_{\eps} u_{\eps})&=& \tilde\phi(t^{+}_{\eps})-\mu\f{{(t^{+}_{\eps})}^{q+1}}{q+1}|u_{\eps}|^{q+1}_{L^{q+1}(\Om)}\\
&\leq& \sup_{t\geq 0}\tilde\phi(t)-\mu\f{(t_0)_{\eps}^{q+1}}{q+1}|u_{\eps}|^{q+1}_{L^{q+1}(\Om)}.
\Eea
Substituting the value of $(t_0)_{\eps}$ and using Sobolev inequality \eqref{in:S}, we have $$\mu\f{(t_0)_{\eps}^{q+1}}{q+1}\geq \f{\mu}{q+1}\bigg(\f{1-q}{2^*-q-1}S_s\bigg)^\f{(N-2s)(q+1)}{4s}=k_8.$$
Consequently, 
\be\lab{4-8}
\sup_{t\in\R}I_{\mu}(t u_{\eps})\leq \sup_{t\geq 0}\tilde\phi(t)-k_8|u_{\eps}|^{q+1}_{L^{q+1}(\Om)}.
\ee
Using elementary analysis, it is easy to check that $\tilde\phi$ attains it's maximum at the point $\tilde t_0= \bigg(\f{||u_{\eps}||^2}{|u_{\eps}|^{2^*}_{L^{2^*}(\Om)}}\bigg)^\f{1}{2^*-2}$ and $\tilde\phi(t_0)=\f{s}{N}\bigg(\f{||u_{\eps}||^2}{|u_{\eps}|^{2}_{L^{2^*}(\Om)}}\bigg)^\f{N}{2s}$. Moreover, from Proposition 21 and Proposition 22 of \cite{SerVal2}, it follows 
$$||u_{\eps}||^2=S_s^\f{N}{2s}+O(\eps^{N-2s}), \quad
\int_{\Rn}|u_{\eps}|^{2^*}dx=S_s^\f{N}{2s}+O(\eps^N). $$
As a result, 
\be\lab{4-9}
\tilde\phi(t_0) \leq \f{s}{N}\bigg[\f{S_s^\f{N}{2s}+O(\eps^{N-2s})}{(S_s^\f{N}{2s}+O(\eps^N))^\f{2}{2^*}} \bigg]^\f{N}{2s}
\leq \f{s}{N}\bigg[\f{S_s^\f{N}{2s}+O(\eps^{N-2s})}{(S_s^\f{N}{2s})^\f{2}{2^*}} \bigg]^\f{N}{2s}\leq \f{s}{N}S_s^\f{N}{2s}+C\eps^{\f{(N-2s)N}{2s}}.
\ee
In the last inequality we have used the fact that $\eps>0$ is arbitrary small. Substituting back \eqref{4-9} into \eqref{4-8}, completes the proof.
\end{proof}

\vspace{2mm}

{\bf Proof of Proposition \ref{p:limit}}: Note that, for fixed $a$ and $b$, $I_{\mu}\big(\eta(a w_1-bu_{\eps})\big)\to -\infty$ as $|n|\to\infty$. Therefore $\sup_{a\geq 0,\ b\in\R} I_{\mu}(a w_1-bu_{\eps})$ exists and supremum will be attained in $a^2+b^2\leq R^2$, for some large $R>0$. Thus it is enough to estimate  $I_{\mu}(a w_1-b u_{\eps})$ in  $\{(a,b)\in \R^{+}\times\R: a^2+b^2\leq R^2 \}$. Using elementary inequality, there exists $d(m)>0$ such that 
\be\lab{in:ele}
|a+b|^m\geq |a|^m+|b|^m-d(|a|^{m-1}|b|+|a||b|^{m-1}) \quad\forall\quad a,\  b\in\R, \ m>1.
\ee 
Therefore,  $a^2+b^2\leq R^2$ implies
\Bea
I_{\mu}(a w_1-bu_{\eps})&\leq& \f{1}{2}||a w_1||^2-ab\<w_1, u_{\eps}\>+\f{1}{2}||b u_{\eps}||^2\\
&-&\f{1}{2^*}\Iom|a w_1|^{2^*}dx-\f{1}{2^*}\Iom|b u_{\eps}|^{2^*}dx\\
&-&\f{\mu}{q+1}\Iom|a w_1|^{q+1}dx-\f{\mu}{q+1}\Iom|b u_{\eps}|^{q+1}dx\\
&+&C\displaystyle\left(\Iom|aw_1|^{2^*-1}|b u_{\eps}|dx+\Iom|a w_1||b u_{\eps}|^{2^{*}-1}dx\right)\\
&+&C\displaystyle\left(\Iom|aw_1|^{q}|b u_{\eps}|dx+\Iom|a w_1||b u_{\eps}|^{q}dx\right)\\
&=& I_{\mu}(a w_1)+I_{\mu}(b u_{\eps})-ab\mu\Iom |w_1|^{q-1}w_1u_{\eps}dx\\
&-&ab\Iom |w_1|^{2^*-2}w_1u_{\eps}dx\\
&+&C\displaystyle\left(\Iom|w_1|^{2^*-1}| u_{\eps}|dx+\Iom| w_1|| u_{\eps}|^{2^{*}-1}dx\right)\\
&+&C\displaystyle\left(\Iom|w_1|^{q}| u_{\eps}|dx+\Iom| w_1|| u_{\eps}|^{q}dx\right).
\Eea
Using Lemmas \ref{l:4i}, \ref{l:4vi} and \ref{l:4vii} we estimate in $a^2+b^2\leq R^2$,
\be
I_{\mu}(a w_1-bu_{\eps})\leq \tilde\al_{\mu}^{-}
+\f{s}{N}S_s^\f{N}{2s} - k_8|u_{\eps}|^{q+1}
+ C(\eps^{\f{(N-2s)N}{2s}}+\eps^\f{N-2s}{4}+\eps^\f{(N-2s)q}{4}+\eps^\f{N+2s}{4})\no.
\ee
Since $N>2s$ and $q\in(0,1)$, clearly $\eps^{(\f{N-2s}{4})q}$ is the dominating term among all the terms inside the bracket.  For the term $k_8|u_{\eps}|^{q+1}$, we invoke  Lemma \ref{l:4ii}. Therefore when $\f{2s}{N-2s}<q<1$, we have
\Bea  I_{\mu}(a w_1-bu_{\eps}) &\leq&   \tilde\al_{\mu}^{-}
+\f{s}{N}S_s^\f{N}{2s} - k_9\eps^{\f{N}{2}-(\f{N-2s}{4})(q+1)}\\
&\quad&+ C(\eps^{\f{(N-2s)N}{2s}}+\eps^\f{N-2s}{4}+\eps^\f{(N-2s)q}{4}+\eps^\f{N+2s}{4})\Eea
This in turn implies, when  $\f{1}{2}(\f{N+2s}{N-2s})<q<1$ and $N>6s$, $\eps^{\f{N}{2}-(\f{N-2s}{4})(q+1)}$ should be the dominating one among all the $\eps$ terms and hence in this case, taking $\eps>0$ to be small enough, we obtain  
$$\sup_{a\geq 0, b\in\R} I_{\mu}(a w_1-bu_{\eps})<   \tilde\al_{\mu}^{-}+\f{s}{N}S_s^\f{N}{2s}.$$
\hfill{$\square$}

\vspace{4mm}

{\bf Proof of Theorem \ref{t:4i}}: Define $\mu_0:=min\{\tilde\mu, \mu_{*} \}$ and 
\be\lab{c-2} c_2: \inf_{u\in \mathcal{N}^{-}_* } I_{\mu}(u),\ee where  
\be\lab{eq:N-star}
\mathcal{N}^{-}_*:=\mathcal{N}^{-}_1\cap \mathcal{N}^{-}_2.
\ee
Let $\mu\in (0, \mu_0)$. Using Ekland's variational principle and similar to the proof of Theorem \ref{t:4ii}, we obtain a  sequence $\{u_n\}\in \mathcal{N}^{-}_{*}$ satisfying 
$$I_{\mu}(u_n)\to c_2, \quad I'_{\mu}(u_n)\to 0 \quad\text{in}\quad (X_0)'.$$ 
Thus $\{u_n\}$ is a (PS) sequence at level $c_2$. From Lemma \ref{l:6-i}, it follows that there exists $a>0$ and $b\in R$ such that $a w_1-bu_{\eps}\in \mathcal{N}^{-}_{*}$. Therefore Proposition \ref{p:limit} yields 
\be\lab{c2c1}
c_2< \tilde\al_{\mu}^{-}+\f{s}{N}S_s^\f{N}{2s}.
\ee
{\bf Claim 1:} There exists two positive constants $c, C$ such that 
$ 0< c \leq \norm{u_n^{\pm}} \leq C$. \\
To see this, we note that $\{u_n\} \subset \mathcal{N}_*^{-}\subset \mathcal{N}_1^{-}$.  Therefore using \eqref{4-13},  Claim 2 and Claim 3 of the proof of Theorem \ref{t:4ii},  we have $ \norm{u_n^{\pm}} \leq C$ and $\norm{u_n^{-}}\geq c$. To show  $\norm{u_n^{+}} \geq a$
for some $a>0$, we use method of contradiction. Assume up to a subsequence  $\norm{u_n^{+}} \to 0$ as $n \to \infty$. 
This  together with Sobolev embedding implies $|u_n^+|_{L^{2^*}(\Om)} \to 0$. On the other hand, $u_n^{+} \in \mathcal{N}^{-}$ implies $(1-q)\norm{u_n^{+}}^2-(2^*-q-1)|u_n^+|^{2^*}_{L^{2^*}(\Om)}<0$. Therefore by \eqref{in:S}, we have 
\begin{align} 
S_s\leq \frac{\norm{u_n^{+}}^2}{|u_n^+|^{2}_{L^{2^*}(\Om)}}<\f{2^*-q-1}{1-q}|u_n^+|^{2^*-2}_{L^{2^*}(\Om)},\no
\end{align}
which is a contradiction to the fact that $|u_n^+|_{L^{2^*}(\Om)} \to 0$. Hence the claim follows.

\vspace{2mm}

Going to a subsequence if necessary we have
\be\lab{4-22}
u_n^{+} \deb \eta_1,\,\, u_n^{-} \deb \eta_2 \quad\mbox{in} \quad X_0. 
\ee
{\bf Claim 2}: $\eta_1 \nequiv 0,\,\,\eta_2 \nequiv 0. $\\
 Suppose not, that is $\eta_1 \equiv 0. $ Then by compact embedding, 
$u_n^{+} \to 0$ in $L^{q+1}(\Om)$. 
Moreover, $u_n^{+} \in \mathcal{N}^{-}\subset \mathcal{N}$, implies $\<I'_\mu(u_n^{+}),u_n^{+}\>=0$. As a consequence, 
$$\norm{u_n^{+}}^2-|u_n^{+}|^{2^*}_{L^{2^*}(\Om)}=\mu|u_n^{+}|^{q+1}_{L^{q+1}(\Om)}=o(1).$$
So we have $|u_n^{+}|^{2^*}_{L^{2^*}(\Om)}=\norm{u_n^{+}}^2+o(1). $ This together with $\norm{u_n^{+}} \geq c$  implies 
$$\frac{|u_n^{+}|^{2^*}_{L^{2^*}(\Om)}}{\norm{u_n^{+}}^2} \geq 1+ o(1). $$
This along with Sobolev embedding gives $|u_n^{+}|^{2^*}_{L^{2^*}(\Om)} \geq S_s^{N/2s}+o(1)$.
Thus we have, 
\be\lab{4-20}
 I_\mu(u_n^{+})= \frac{1}{2}\norm{u_n^{+}}^2-\frac{1}{2^*}|u_n^{+}|^{2^*}_{L^{2^*}(\Om)}+o(1)\geq \frac{s}{N}S_s^{N/2s}+o(1).
 \ee
Moreover,  $u_n \in \mathcal{N}_*^{-}$ implies $-u_n^{-} \in \mathcal{N}^{-}$. Therefore using the given condition on $\ba_2$, we get 
\be\lab{4-21}
I_\mu(-u_n^{-}) \geq \ba_2 \geq \tilde\al_{\mu}^{-}.
\ee
 Also it follows  \ $I_\mu(u_n^{+})+I_\mu(-u_n^{-}) \leq I_\mu(u_n)=c_2+o(1)$ (see \eqref{4-14'}). Combining this along with \eqref{4-21} and \eqref{c2c1}, we obtain
$$I_\mu(u_n^{+}) \leq c_2-  \tilde\al_{\mu}^{-}+o(1)<\frac{s}{N}S_s^{N/2s},$$
which is a contradiction to \eqref{4-20}. Therefore, $\eta_1 \neq 0.$ Similarly, $\eta_2 \neq 0$ and this proves the claim.

\vspace{2mm}

Set $w_2 :=\eta_1-\eta_2.$

{\bf Claim 3:} $w_2^{+}=\eta_1$ and $w_2^{-}=\eta_2$ a.e..\\ 
To see the claim we observe that $\eta_1\eta_2=0$ a.e. in $\Om$. Indeed,
\bea\lab{4-22'}
\displaystyle |\Iom \eta_1\eta_2 dx| &=&|\Iom (u_n^+-\eta_1)u_n^- dx+\Iom \eta_1(u_n^--\eta_2)dx|\no\\ 
&\leq& |u_n^+-\eta_1|_{L^2(\Om)}|u_n^{-}|_{L^2(\Om)}+|\eta_1|_{L^2(\Om)}|u_n^{-} -\eta_2|_{L^2(\Om)}
\eea
By compact embedding we have $u_n^{+}\to \eta_1$ and $u_n^{-}\to \eta_2$ in $L^2(\Om)$. Therefore using claim 1, we pass the limit in \eqref{4-22'} and obtain $\Iom \eta_1\eta_2 dx=0$. Moreover  by \eqref{4-22}, $\eta_1, \ \eta_2\geq 0$ a.e.. Hence $\eta_1\eta_2=0$ a.e. in $\Om$. We have $w_2^+-w_2^-=w_2=\eta_1-\eta_2$. It is easy to check that $w_2^+\leq \eta_1$ and $w_2^{-}\leq \eta_2$.  To show that equality holds a.e. we apply the method of contradiction. Suppose, there exists $E\in\Om$ such that $|E|>0$ and $0\leq w_2^+(x)<\eta_1(x) \ \forall\ x\in E$.  Therefore $\eta_2=0$ a.e. in $E$ by the observation that we made. Hence $w_2^{+}(x)-w_2^{-}(x)=\eta_1(x)$  a.e. in $E$.
Clearly $w_2^{-}(x)\not>0$ a.e., otherwise $w_2^{+}(x)=0$ a.e. and that would imply $\eta_1(x)=-w_2^{-}(x)<0$ a.e, which is not possible since $\eta_1>0$ in $E$. Thus $w_2^{-}(x)=0$. This yields $\eta_1(x)=w_2^{+}(x)$ a.e. in $E$, which is a contradiction. Thus the claim follows.

\vspace{2mm}

Therefore $w_2$ is sign changing in $\Om$ and $u_n \deb w_2$ in $X_0$. 
Moreover, $I'_{\mu}(u_n)\to 0$ in $(X_0)'$ implies
\Bea &\displaystyle\int_{\R^{2N}}\f{(u_n(x)-u_n(y))(\phi(x)-\phi(y))}{|x-y|^{N+2s}}dxdy-\mu\Iom|u_n|^{q-1}u_n\phi dx-\Iom |u_n|^{2^*-2}u_n\phi dx\\
&=o(1),\Eea
for every $\phi\in X_0$.
Passing  the limit using Vitali's convergence theorem via H\"older's inequality we obtain $\<I'_\mu(w_2),\phi\>=0$.  As a result, $w_2$ is a sign changing weak solution to $(\mathcal{P})$.
\hfil{$\square$}

\begin{lemma}\label{l:6-i}
Let $u_{\eps}$ be as defined in \eqref{u-eps} and $w_1$ be a positive solution of $(\mathcal{P})$ for which $\tilde\al_{\mu}^{-}$ is achieved, when $\mu\in (0,\mu_*)$. Then there exists $a, \ b \in \R, \ a \geq 0$ such that $aw_1-bu_\var \in \mathcal{N}^{-}_{*}$, where  $\mathcal{N}^{-}_{*}$ is defined as in \eqref{eq:N-star}.
\end{lemma}

\begin{proof}
 We will show that there exists $a>0, \ b \in \R$ such that 
 $$a(w_1-bu_\var)^{+} \in \mathcal{N}^{-} \quad\mbox{and}\quad -a(w_1-bu_\var)^{-} \in N^{-}. $$
 Let us denote $\bar r_1=\inf_{x \in \Om}\frac{w_1(x)}{u_\var(x)}, \,\, \bar r_2=\sup_{x \in \Om}\frac{w_1(x)}{u_\var(x)}$.\\
 As both $w_1$ and $u_\var$ are positive in $\Om$, we have $\bar r_1 \geq 0$ and  $\bar r_2$ can be $+\infty$.
Let $r \in (\bar r_1,\bar r_2)$. Then  $w_1,u_\var \in X_0$ implies $(w_1-ru_\var) \in X_0$ and $(w_1-ru_\var)^{+} \nequiv 0$.
Otherwise, $(w_1-ru_\var)^{+} \equiv 0$ would imply $\bar r_2 \leq r$, which is not possible. Define $v_r := w_1-ru_\var$. Then
 $0\not\equiv v_r^{+}\in X_0$ (since for any $u\in X_0$, we have $|u|\in X_0$). Similarly, $0\not\equiv v_r^{-}\in X_0$. Therefore, by lemma $\ref{l:4-iii}$ there exists $0<s^{+}(r)<s^{-}(r)$ such that $s^{+}(r)v^{+}_r \in \mathcal{N}^{-}$, 
 and $-s^{-}(r)(v_r^{-}) \in \mathcal{N}^{-} $.
Let us consider the functions $s^{\pm}: \R \to (0,\infty)$  defined as above.\\
\textit{Claim}: The functions $r\mapsto s^{\pm}(r)$ are continuous and $$\lim_{r \to \bar r_1^+}s^{+}(r)=t^{+}(v^{+}_{\bar r_1}) \quad\text{and}\quad
\lim_{r \to \bar r_2^-}s^{+}(r)=+\infty,$$ where the function $t^{+}$ is same as defined  in lemma \ref{l:4-iii}.\\
To see the claim, choose $r_0 \in (\bar r_1,\bar r_2)$ and $\{r_n\}_{n \geq 1} \subset (\bar r_1,\bar r_2)$  such that $r_n \to r_0$ as $n \to \infty$. We need to show 
that $s^{+}(r_n)\to s^{+}(r_0)$ as $n \to \infty$.  Corresponding to $r_n$ and $r_0$, we have  $v_{r_n}^{+}=(w_1-r_nu_\var)^{+}$ and $v_{r_0}^{+}=(w_1-r_0u_\var)^{+}$.
By lemma $\ref{l:4-iii}$. we note that $s^{+}(r)=t^{+}(v^{+}_r)$.
Let us define the function 
\Bea
F(s, r)& :=&s^{1-q}\norm{(w_1-ru_{\eps})^+}^2-s^{2^*-q-1}|(w_1-ru_{\eps})^+|^{2^*}_{L^{2^*}(\Om)} \\
&-&\mu|(w_1-ru_{\eps})^+|^{q+1}_{L^{q+1}(\Om)}\\
&=& \phi(s, r)-\mu|(w_1-ru_{\eps})^+|^{q+1}_{L^{q+1}(\Om)},
\Eea
where $$\phi := s^{1-q}\norm{(w_1-ru_{\eps})^+}^2-s^{2^*-q-1}|(w_1-ru_{\eps})^+|^{2^*}_{L^{2^*}(\Om)},$$ is defined similar to \eqref{eq:phi} (see Lemma \ref{l:4-iii}). 
Doing the similar calculation as in lemma $\ref{l:4-iii}$, we obtain that for any fixed $r$, the function $F(s, r)$ has only two zeros $s=t^{+}(v_r^+)$ and $s=t^{-}(v_r^+)$ (see \eqref{4-5}). Consequently $s^{+}(r)$ is the largest $0$ of $F(s, r)$ for  any fixed $r$. As $r_n \to r_0$ we have $v^{+}_{r_n} \to v^{+}_{r_0}$ in $X_0$ .
Indeed, by straight forward computation it follows $v_{r_n}\to v_{r_0}$ in $X_0$. Therefore $|v_{r_n}|\to |v_{r_0}|$ in $X_0$. This in turn implies $v_{r_n}^+\to v_{r_0}^+$ in $X_0$.
Hence $||v_{r_n}^+||_{X_0}\to ||v_{r_0}^+||_{X_0}$. Moreover by Sobolev inequality, we have $|v_{r_n}^+|_{L^{2^*}(\Om)}\to 
|v_{r_0}^+|_{L^{2^*}(\Om)}$ and  $|v_{r_n}^+|_{L^{q+1}(\Om)}\to |v_{r_0}^+|_{L^{q+1}(\Om)} $. As a result, we have
$F(s, r_n)\to F(s, r_0)$ uniformly. Therefore an elementary analysis yields $s^+(r_n)\to s^+(r_0)$.

Moreover, $\bar r_2\geq \f{w_1}{u_\eps}$ implies $w_1-\bar r_2u_{\eps}\leq 0$. As a consequence $r\to \bar r_2^-$ implies $(w_1-r u_{\eps})^+\to 0$ pointwise. Moreover, since $|(w_1-r u_{\eps})^+|_{L^{\infty}(\Om)}\leq |w_1|_{L^{\infty}(\Om)}$, using dominated convergence theorem we have\\ $|(w_1-r u_{\eps})^+|_{L^{2^*}(\Om)}\to 0$.
From the analysis in Lemma \ref{l:4-iii}, for any $r$, we also have $s^+(r)>t_0(v_r^+)$, where function $t_0$ is defined as in \eqref{t-0}, which is the maximum point of $\phi(.,r)$.
Therefore it is enough to show that $\lim_{r\to \bar r_2^-}t_0(v_r^+)=\infty$. Applying \eqref{in:S} in the definition of $t_0(v_r^+)$ we get
$$t_0(v_r^+)=\bigg(\f{(1-q)||v_r^+||^2}{(2^*-1-q)|v_r^+|_{L^{2^*}(\Om)}^{2^*}}\bigg)^\f{1}{2^*-2}\geq\bigg(\f{S_s(1-q)}{2^*-1-q}\bigg)^\f{1}{2^*-2}|v_r^+|_{L^{2^*}(\Om)}^{-1}.$$ 
Thus $\lim_{r\to \bar r_2^-}t_0(v_r^+)=\infty$.  

Similarly proceeding  as above we can show that if $r\to \bar r_1^-$ then $v_r^+\to v_{\bar r_1}$ and 
$\lim_{r \to \bar r_1^+}s^{+}(r)=t^{+}(v^{+}_{\bar r_1})$ and
$$\lim_{r \to r_1^{+}}s^{-}(r)=+\infty,\,\, \lim_{r\to r_2^{-}}s^{-}(r)=t^{+}(v_r^-)<+\infty. $$
The continuity of $s^{\pm}$ implies that there exists $b \in (\bar r_1,\bar r_2)$ such that
$s^{+}(r)=s^{-}(r)=a>0$.
Therefore, $$a(w_1-bu_\var^{+}) \in \mathcal{N}^{-} \quad\mbox{and}\quad -a(w_1-bu_\var^{-}) \in \mathcal{N}^{-},$$
that is, the function $a(w_1-bu_\var) \in \mathcal{N}_*^{-}$ and this completes the proof.

\end{proof}

\vspace{4mm}

{\bf Proof of Theorem \ref{thm.2}}: Define $\mu^*=\min\{\mu_*,\tilde\mu, \mu_0, \mu_1\}$. 
Combining Theorem \ref{t:4ii} and Theorem \ref{t:4i}, we complete the proof of this theorem for $\mu\in(0,\mu^*)$ .
\hfil{$\square$}

\section{\bf Proof of Theorem \ref{thm.3}} 

Before starting the proof we like to remark that when $\mu\geq 0,\ \la>0$, Theorem \ref{thm.3} (a) also follows from \cite[Theorem 1]{BMS}. Here we give a proof which covers the entire range mentioned in Theorem \ref{thm.3}. 

\proof  (a)We assume $\mu\in\R$ and $\la>0$. We prove part (a) using Fountain theorem \ref{t:F}. Energy functional corresponding to $(\mathcal{P}_K)$ is defined by $I_{\mu}^{\la}$ (see \eqref{I-mu-la}). We need to verify that $I_{\mu}^{\la}$ satisfies (A1)-(A4) of Theorem \ref{t:F}.
We choose $X_j, Y_j, Z_j$  as in \eqref{YZ_k} and $G:=\Z/2$. Therefore, (A1) is satisfied. 

\vspace{2mm}

Next to check (A2) holds, we observe that,  
$$I_{\mu}^{\la}(u)\leq\frac{1}{2}{||u||_{X_0}}^2+\f{|\mu|}{q+1}|u|_{L^{q+1}(\Om)}^{q+1}-\frac{\la}{p+1}|u|_{L^{p+1}(\Om)}^{p+1}.$$
Since on the finite dimensional space $Y_k$ all the norms are equivalent, $\la>0$ and $1<q+1<2<p+1$, it is easy to see that (A2) is satisfied if we choose $\rho_k>0$ large enough.
 
\vspace{2mm} 

To see (A3) holds,  we observe that 
\be\label{5-1}
I_{\mu}^{\la}(u)\geq \frac{\norm{u}_{X_0}^2}{2}-\frac{|\mu|}{q+1}\Iom |u|^{q+1} dx-\frac{\la}{p+1}\Iom |u|^{p+1}dx. \ee
Applying H\"{o}lder inequality followed by Youngs inequality we obtain 
$$\Iom |u|^{q+1} dx\leq \f{q+1}{p+1}\Iom |u|^{p+1}dx +\f{p-q}{p+1}|\Om|.$$
 Substituting back in \eqref{5-1}, we obtain
$$I_{\mu}^{\la}(u) \geq \frac{1}{2}\norm{u}_{X_0}^2-\bigg(\frac{|\mu|}{p+1}+\frac{\la}{p+1}\bigg)\Iom |u|^{p+1} - \frac{(p-q)|\mu|}{(p+1)(q+1)}|\Om|.$$
Define
$$\beta_k:=\sup_{{u \in Z_k},\ {\norm{u}_{X_0}=1}} |u|_{L^{p+1}(\Om)}.$$ Hence on $Z_k$ we have 
$$I_{\mu}^{\la}(u) \geq \frac{1}{2}\norm{u}_{X_0}^2-\frac{(\la+|\mu|)\beta_k^{p+1}}{p+1}\norm{u}^{p+1}- \frac{(p-q)|\mu|}{(p+1)(q+1)}|\Om|. $$
Choosing $r_k^{1-p}=(\la+|\mu|)\beta_k^{p+1}$, we have, for  $u \in Z_k$ and $\norm{u}_{X_0}=r_k, $ 
$$I_{\mu}^{\la}(u) \geq \bigg(\frac{1}{2}-\frac{1}{p+1}\bigg) r_k^2-\frac{(p-q)|\mu|}{(p+1)(q+1)}|\Om|. $$
Lemma \ref{beta_k} yields $\beta_k \to 0 $ and hence  $r_k \to \infty$ as $k \to \infty. $ Therefore (A3) is satisfied.

\vspace{2mm}

In order to verify (A4), let $\{u_n\}\subset X_0$ such that 
$$I_{\mu}^{\la}(u_n)\to c \quad\text{and}\quad (I_{\mu}^{\la})'(u_n)\to 0 \quad\text{in}\quad (X_0)',$$
where $c>0$ and $(X_0)'$ denotes the dual space of $X_0$. Following the same calculation as in Theorem \ref{thm.1}, we get $\{u_n\}$ is bounded in $X_0$ and there exists $u\in X_0$ such that up to a subsequence $u_n\deb u$ in  $X_0$ and $u_n\to u$ in $L^r(\Rn)$ for every $r \in[1, 2^*)$. Since $\<(I_{\mu}^{\la})'(u_n), v\>=0$ for every $v$ in $X_0$, passing the limit using Vitali's convergence theorem, it follows $\<(I_{\mu}^{\la})'(u), v\>=0$ for every $v$ in  $X_0$. Therefore
\Bea
o(1) &=&\<(I_{\mu}^{\la})'(u_n)-(I_{\mu}^{\la})'(u),u_n-u\> \\
&=&\norm{u_n-u}^2 \\
&-&\mu\Iom{(|u_n|^{q-1}u_n-|u|^{q-1}u)(u_n-u)}dx\\
&-& \la\Iom{(|u_n|^{p-1}u_n-|u|^{p-1}u)(u_n-u)} dx.
\Eea
Again, passing the limit by Vitali, we obtain $u_n\to u$ in $X_0$. Hence, (A4) is satisfied. Therefore by Theorem \ref{t:F}, it follows that $(\mathcal{P}_K)$ has a sequence of nontrivial solution $\{w_k\}_{k \geq 1}$ such that $I_\mu^\la(w_k) \to \infty$ as 
 $k \to \infty. $ Furthermore, if $\la >0,\ \mu \geq 0, $ then $I_{\mu}^{\la}(w_k)\leq ||w_k||^2$ and thus $\norm{w_k}_{X_0} \to \infty$ as $k \to \infty$.

\vspace{2mm}

(b) This part follows from Theorem \ref{t:DF}. We can proceed along the same line of proof of Theorem \ref{thm.1} to show (D1)-(D3) of Theorem \ref{t:DF} are satisfied.  
To check the assertion (D4), we consider a sequence $\{u_{r_j}\} \subset X_0$ such that as 
\be
 \{u_{r_j}\} \in Y_{r_j},  \quad
I_{\mu}^{\la}(u_{r_j}) \to c, \quad (I_{\mu}^{\la})|'_{Y_{r_j}}(u_{r_j}) \to 0 \quad\text{as}\quad r_j \to \infty\no.
\ee
We can prove exactly in the same way as in Theorem \ref{thm.1} that $\{u_n\}$ is a bounded PS sequence in $X_0$ at level $c$. Therefore, it is easy to conclude, as in part (a) that $u_n$ converges strongly in $X_0$. Hence (D4) is also satisfied and as a result by Theorem \ref{t:DF}, we conclude $(\mathcal{P}_K)$ has a sequence of nontrivial solutions $\{v_k\}_{k \geq 1}$ such that $c_k:= I_\mu^\la(v_k)< 0$ and $c_k\to 0$ as 
 $k \to \infty. $ Using $\<(I_{\mu}^{\la})'(u_k), u_k\>=0$ in the definition of $I_{\mu}^{\la}(u_k)$, we have
 $$\mu\bigg(1-\f{2}{q+1}\bigg)\Iom|u|^{q+1}dx+\la\bigg(1-\f{2}{p+1}\bigg)\Iom|u|^{p+1}dx=2c_k<0.$$
Therefore, if $\mu >0,\ \la \leq 0, $ then 
$$0\leq-\la\bigg(1-\f{2}{p+1}\bigg)\Iom|u|^{p+1}dx=-2c_k+\mu\bigg(1-\f{2}{q+1}\bigg)\Iom|u|^{q+1}dx,$$
 since $1<q+1<2<p+1$. This implies, $-2c_k\geq -\mu\big(1-\f{2}{q+1}\big)\Iom|u|^{q+1}dx$. Hence $\Iom |u_k|^{q+1}dx\leq \f{-2c_k q}{\mu(2-q)}$. Moreover, $\<(I_{\mu}^{\la})'(u_k), u_k\>=0$ implies  $$||u_k||^2=\mu\Iom|u_k|^{q+1}dx+\la\Iom |u_k|^{p+1}dx\leq \mu\Iom|u_k|^{q+1}dx\leq  \f{-2c_k q}{2-q}\to 0,$$
as $k\to\infty$. This completes the proof.
 \hfil{$\square$}
 
\section{\bf A related variational problem}
In this section we consider a related problem that can be solved by doing the similar type of analysis that we did in Section 3. More precisely we consider the following problem:
\begin{align} \label{eq-HS}
\begin{cases}
  \left(-\Delta\right)^s u-\frac{\al u}{\abs{x}^{2s}} = \frac{|u|^{2^*(t)-2}u}{\abs{x}^t} + \mu |u|^{q-1}u &\quad\mbox{in}\quad \Omega,\\
 u=0&\quad\mbox{in}\quad\mathbb{R}^N\setminus\Omega,
 \end{cases}
\end{align}
where $N>2s$, $\Omega$ is an open, bounded domain in $\Rn$ with smooth boundary, $0 \leq t < 2s ,\,
0<q<1,\,2^*(t)=\frac{2(N-t)}{N-2s}, \,
\al <\al_H:=2^{2s} \frac{\Ga^2(\frac{N+2s}{4})}{\Ga^2(\frac{N-2s}{4})}$ is the best fractional Hardy constant on $\Rn$. Thanks to the following fractional Hardy inequality :
\be\lab{in:HS}
\al_H\int_{\Rn}\f{|u|^2}{|x|^{2s}}dx\leq \displaystyle\int_{\Rn}|(-\De)^\f{s}{2}u|^2 dx,
\ee
which was proved by Herbst \cite{H} (see also \cite{Ghouss}), \\ $\bigg(\displaystyle\int_{\Rn}|(-\De )^\f{s}{2}u|^2dx-\al\displaystyle\int_{\Om}\frac{\abs{u(x)}^2}{\abs{x}^{2s}}\bigg)^{\frac{1}{2}}$ is a norm equivalent to the norm \eqref{norm-2} in $X_0(\Om)$. Interpolating the above Hardy inequality with \eqref{in:S} and followed by simple calculation, we have the following fractional Hardy-Sobolev inequality
$$C\displaystyle\left(\int_{\Om}\f{|u|^{2^*(t)}}{|x|^t}dx\right)^\f{2}{2^*(t)} \leq\int_{\Rn}|(-\De )^\f{s}{2}u|^2dx-\al\int_{\Om}\frac{\abs{u(x)}^2}{\abs{x}^{2s}}.$$ Therefore we can define the quotient $S_s(\al)>0$ as follows
\be\lab{S-al}
S_s(\al):=\inf_{u\in X_0,\ u\not=0} \f{\displaystyle\int_{\Rn}|(-\De )^\f{s}{2}u|^2dx-\al\int_{\Om}\frac{\abs{u(x)}^2}{\abs{x}^{2s}}  }{\displaystyle\left(\int_{\Om}\f{|u|^{2^*(t)}}{|x|^t}dx\right)^\f{2}{2^*(t)}}.
\ee
The following theorem regarding existence of infinitely many nontrivial solutions for fractional Hardy-Sobolev type equation can be proved in the spirit of
theorem $\ref{thm.1}. $ 
\begin{theorem} \label{thm.4}
Let $\Om$ be a bounded domain in $\Rn$ with smooth boundary, $N>2s. $Then there exists $\mu^*>0$ such that for all $\mu \in (0,\mu^*)$, problem $(\ref{eq-HS})$ has a sequence of non-trivial solutions $\{u_n\}_{n \geq 1}$ such that $I(u_n) <0$ and $I(u_n) \to 0$ as $n \to \infty$ where $I(.)$ is the corresponding energy functional associated with $(\ref{eq-HS}). $
\end{theorem} 
In order to prove this theorem one essentially needs to verify an argument similar to \eqref{3:c}, where RHS of \eqref{3:c} should be replaced by $\f{2s-t}{2(N-t)}S_s(\al)^\f{N-t}{2s-t}-k\mu^\f{2^*(t)}{2^*(t)-q-1}$ and this would follow by the similar type of arguments as in the proof of Theorem \ref{thm.1}. 

\section{\bf Appendix} 

\begin{lemma}\label{l:6-ii}
Let $g_n$ be as in \eqref{g_n} in the Theorem \ref{t:4ii} and $v\in X_0$ such that $||v||=1$. Then there exists $\mu_1>0$ such that, $\mu\in (0,\mu_1)$ implies $\<g_n'(0),v\>$ is uniformly bounded in $X_0$. 
\end{lemma}
\begin{proof}
 In view of lemma $\ref{l:4-v}$ we have,
$$\<g'_n(0),v\>=\displaystyle\frac{2\<u_n,v\>-2^*\displaystyle\Iom |u_n|^{2^*-2}u_nv-(q+1)\mu\displaystyle\Iom |u_n|^{q-1}u_nv}{(1-q)\norm{u_n}^2-
       (2^*-q-1)|u_n|^{2^*}_{L^{2^*}(\Om)}}.$$
Using Claim 2 in theorem \ref{t:4ii}, there exists $C>0 $ such that $\norm{u_n} \leq C$ for all $n \geq 1. $
Therefore applying H\"older inequality followed by  \eqref{in:S}, we have\\
$|\<g'_n(0),v\>|\leq \frac{C\norm{v}}{|(1-q)\norm{u_n}^2-(2^*-q-1)|u_n|^{2^*}_{L^{2^*}(\Om)}|}. $
Hence it is enough to show   
$$|(1-q)\norm{u_n}^2-(2^*-q-1)|u_n|^{2^*}_{L^{2^*}(\Om)}|>C,$$ for some $C>0$ and $n$ large.
Suppose it does not hold. Then up to a subsequence $$(1-q)\norm{u_n}^2-(2^*-q-1)|u_n|^{2^*}_{L^{2^*}(\Om)}=o(1)\quad\mbox{as}\quad n\to \infty. $$
Hence,
\begin{align} \label{*1}
\norm{u_n}^2=\frac{2^*-q-1}{1-q}|u_n|^{2^*}_{L^{2^*}(\Om)}+o(1)\quad\mbox{as}\quad n\to \infty.
\end{align}
Combining  the above expression along with the fact that $u_n \in \mathcal{N}$, we obtain
\begin{align} \label{*4}
\mu|u_n|^{q+1}_{L^{q+1}(\Om)}=\frac{2-2^*}{1-q}\abs{u_n}^{2^*}_{L^{2^*}(\Om)} +o(1)=\frac{2^*-2}{2^*-1-q}\norm{u_n}^2+o(1).
\end{align}
After applying H\"{o}lder inequality and followed by \eqref{in:S},  expression  \eqref{*4} yields 
\begin{align}\label{*6}
\norm{u_n} \leq \bigg(\mu\frac{2^*-q-1}{2^*-2}|\Om|^{\frac{2^*-q-1}{2^*}}S_s^{-\frac{q+1}{2}} \bigg)^{\frac{1}{1-q}} +o(1).
\end{align}
Combining \eqref{4-13} and Claim 3 in the proof of Theorem \ref{t:4ii}, we have $\norm{u_n}\geq b$, for some $b>0$. Therefore from \eqref{*1} we get
\begin{align}\label{*9}
|u_n|^{2^*}_{L^{2^*}(\Om)} \geq C\quad\mbox{for some constant}\  C>0,\ \mbox{and}\,\, n\,\,\mbox{large enough.}
\end{align}
Define $\psi_\mu: \mathcal{N} \to \R$ as follows:
$$\psi_\mu(u)=k_0\bigg(\frac{\norm{u}^{2(2^*-1)}}{|u|^{2^*}_{L^{2^*}(\Om)}}\bigg)^{\frac{1}{2^*-2}}-\mu|u|^{q+1}_{L^{q+1}(\Om)},$$
where $k_0=\left(\frac{1-q}{2^*-q-1}\right)^{\frac{N+2s}{4s}}\left(\frac{2^*-2}{1-q}\right)$. 
Simplifying  $\psi_{\mu}(u_n)$ using \eqref{*4}, we obtain
\be \label{*7}
\psi_\mu(u_n)=k_0\bigg[\bigg(\f{2^*-q-1}{1-q}\bigg)^{2^*-1}\f{|u_n|^{(2^*-1)2^*}_{L^{2^*}(\Om)}}{|u_n|_{L^{2^*}(\Om)}^{2^*}}\bigg]^\f{1}{2^*-2} - \f{2^*-2}{1-q} |u_n|_{L^{2^*}(\Om)}^{2^*}+o(1)=o(1).
\ee
On the other hand, using H\"older inequality in the definition of $\psi_{\mu}(u_n)$, we obtain
\bea\lab{10*}
 \psi_\mu(u_n)&=&k_0\bigg(\frac{\norm{u_n}^{2(2^*-1)}}{|u_n|^{2^*}_{L^{2^*}(\Om)}}\bigg)^{\frac{1}{2^*-2}}-
 \mu|u_n|^{q+1}_{L^{q+1}(\Om)}\notag \\               
&\geq&k_0\bigg(\frac{\norm{u_n}^{2(2^*-1)}}{|u_n|^{2^*}_{L^{2^*}(\Om)}}\bigg)^{\frac{1}{2^*-2}}-
\mu|\Om|^{\f{2^*-q-1}{2^*}}|u_n|^{q+1}_{L^{2^*}(\Om)}\notag\\
&=&|u_n|^{q+1}_{L^{2^*}(\Om)}\bigg\{k_0\bigg(\frac{\norm{u_n}^{2(2^*-1)}}{|u_n|^{2^*}_{L^{2^*}(\Om)}}\bigg)
^{\frac{1}{2^*-2}}\frac{1}{|u_n|^{q+1}_{L^{2^*}(\Om)}}-\mu|\Om|^{\f{2^*-q-1}{2^*}}\bigg\}.
\eea
Using \eqref{in:S} and \eqref{*6}, we simplify the term $\bigg(\frac{\norm{u_n}^{2(2^*-1)}}{|u_n|^{2^*}_{L^{2^*}(\Om)}}\bigg)
^{\frac{1}{2^*-2}}\frac{1}{|u_n|^{q+1}_{L^{2^*}(\Om)}}$ and obtain
\bea\lab{11*}
\bigg(\frac{\norm{u_n}^{2(2^*-1)}}{|u_n|^{2^*}_{L^{2^*}(\Om)}}\bigg)
^{\frac{1}{2^*-2}}\frac{1}{|u_n|^{q+1}_{L^{2^*}(\Om)}} &\geq& S_s^\f{N+2s}{4s}|u_n|_{L^{2^*}(\Om)}^{-q}\no\\
&\geq& S_s^\f{N+2s(q+1)}{4s}||u_n||^{-q}\no\\
&\geq& S_s^\f{N+2s(q+1)}{4s} \bigg(\mu\frac{2^*-q-1}{2^*-2}|\Om|^{\frac{2^*-q-1}{2^*}}S_s^{-\frac{q+1}{2}} \bigg)^{-\frac{q}{1-q}}.
\eea
Substituting back \eqref{11*} into \eqref{10*} and using \eqref{*9}, we obtain
\Bea
 \psi_\mu(u_n)&\geq& C^{q+1}\bigg [k_0 S^{\f{N+2s(q+1)}{4s}+(\f{1+q}{1-q}\f{q}{2})}\mu^{-\f{q}{1-q}}\big(\frac{2^*-q-1}{2^*-2}|\Om|^{\frac{2^*-q-1}{2^*}}\big)^{-\f{q}{1-q}}-\mu|\Om|^\f{2^*-q-1}{2^*}\bigg]\\
 &\geq& d_0,
\Eea
for some $d_0>0$, $n$ large and $\mu < \mu_1$, where
$\mu_1=\mu_1(k, s, q, N, |\Om|)$. This is a contradiction to \eqref{*7}. Hence the lemma follows. 
\end{proof}

\section{\bf Concluding remarks and questions}
We finish the paper with some remarks and related open questions. We have proved existence of infinitely many nontrivial solutions of the following problem 
\begin{align*}
 \begin{cases}
  \mathcal (-\De)^s u = \mu |u|^{q-1}u + |u|^{p-1}u=0 &\quad\mbox{in}\quad \Omega,\\
  u=0&\quad\mbox{in}\quad\mathbb{R}^N\setminus\Omega,
 \end{cases}
\end{align*}
where $0<s<1$, $0<q<1< p\leq {2^*-1}$, $\mu>0$, $N>2s$ and existence of at least one sign changing solution when $p=2^*-1$, $\f{1}{2}(\f{N+2s}{N-2s})<q<1$ and $N>6s$. Also it is known (see \cite{BCSS, BrCPS, CD}) that there exists $\mu_0>0$ such that, for every $\mu\in(0,\mu_0)$, the above problem admits at least two positive solutions. An
interesting question here is, can we say anything about exact number of positive solutions at least in some bounded domain, for example in the unit ball in $\Rn$? This fact is known in the case of local operator (see \cite{APY, DGP, OS, Tang}).  Answer to this question  can shed some light on the multiplicity questions of sign-changing solutions. 

In our forthcoming paper we study similar kind of problems for larger class of nonlocal equations where the leading term is given by
some nonlinear integro-differential operators; that is, the ones obtained
by replacing the fractional Laplacian in (1.1) with the following nonlinear operator,

\begin{align}
 \mathcal  F(u)(x)=P.V \int_{\mathbb{R}^N}|u(x)-u(y)|^{p-2}(u(x)-u(y))K(x,y)dy,\,\,\,x\in\mathbb{R}^N,\no 
\end{align}
where $K$ is a symmetric kernel of differentiability order $s\in(0,1)$ and $p>1$ with general possibly non smooth coefficients, as considered for instance
in the recent papers by Di Castro et al. \cite{CKP}.
\vspace{4mm}

{\bf Acknowledgement} The authors thank the referee for his/her valuable comments. The first author is supported by the INSPIRE research grant DST/INSPIRE 04/2013/000152 and the second author is supported by the NBHM grant 2/39(12)/2014/RD-II.

  \end{document}